\documentclass[11pt]{amsart}
\usepackage[margin=1.3in]{geometry}

\usepackage{nicefrac}
\usepackage[dvipsnames]{xcolor}
\usepackage{esvect}
\usepackage{graphicx}      
\usepackage{tikz}
\usepackage{pgfplots}
\usepackage{amsmath,amssymb, amsfonts,amsthm}
\usepackage{graphicx}
\usepackage[]{amsaddr}
\usepackage{lmodern}

\newcommand{\R}{\mathbb{R}}

\newcommand{\cU}{\mathcal{U}}

\newcommand{\diag}{\operatorname{diag}}

\newcommand{\C}{\mathbb{C}}

\newcommand{\bC}{\mathbf{C}}

\newcommand{\rL}{\mathrm{L}}

\newcommand{\bfo}{\mathbf{1}}

\usepackage{microtype}

\newcommand{\bD}{{\mathbf{D}}}

\newcommand{\rank}{\operatorname{rank}}

\newcommand{\spec}{{\operatorname{spec}}}

\newcommand{\rrd}{\mathrm{d}}

\newcommand{\bV}{\mathbf{V}}

\newcommand{\bR}{\mathbf{R}}
\newcommand{\bF}{\mathbb{F}}
\newcommand{\bM}{\mathbf{M}}

\makeatletter
\def\blfootnote{\xdef\@thefnmark{}\@footnotetext}
\makeatother

\newcommand{\xc}[1]{\vspace{.1cm}

\noindent {\em #1} }

\newcommand{\mab}[1]{\vspace{.1cm}

\noindent {\textbf{#1 }}} 

\newtheorem{definition}{Definition}
\newtheorem{theorem}{Theorem}[section]
\newtheorem{lemma}{Lemma}[section]
\newtheorem{proposition}[theorem]{Proposition}
\newtheorem{corollary}[theorem]{Corollary}

\newtheorem{problem}{Problem}
\newtheorem*{mainth*}{Main Theorem}

\usepackage{subcaption}

\title{On Infinite-horizon Minimum Energy Control}

\begin{document}

\author[Belabbas]{Mohamed-Ali Belabbas}
\address[Belabbas]{Electrical and Computer Engineering Department and Coordinated Science Laboratory, University of Illinois, Urbana-Champaign.}
\email[Belabbas]{belabbas@illinois.edu}
\author[Chen]{Xudong Chen}
\address[Chen]{Electrical and Systems Engineering, Washington University in St. Louis.}
\email[Chen]{cxudong@wustl.edu}
\thanks{Both authors contributed equally to this manuscript.}

\maketitle
\begin{abstract}
    We address the infinite-horizon minimum energy control problem for linear time-invariant finite-dimensional systems $(A, B)$. We show that the problem admits a solution if and only if $(A, B)$ is stabilizable and $A$ does not have imaginary eigenvalues.
\end{abstract}

\section{Introduction}

Let $\bF$ be either $\R$ or $\C$, the field of real or complex numbers. 
We consider in this paper finite-dimensional linear time-invariant systems:
\begin{equation}\label{eq:linsys}
    \dot x(t) = Ax(t) + Bu(t),
\end{equation}
where $A \in \bF^{n \times n}$, $B\in \bF^{n \times m}$,  $x(t) \in \bF^n$, and $u(t) \in \bF^m$.

Given $x_0, x_1 \in \bF^n$, the {\em finite-horizon  minimum energy control problem} for~\eqref{eq:linsys} is 
\begin{equation}\label{eq:minenergyprobfinT}
    \min_{u \in \mathrm{L}^2([0,T],\bF^m)
    } \eta_T (u):= \int_0^T \|u(t)\|^2 \rrd t \quad \mbox{s.t. } x(0) = x_0 \quad \mbox{and} \quad x(T)= x_1. 
\end{equation}

It is well known~\cite{brockett2015finite, wonham1974linear,kwakernaak1972linear} that if the pair $(A,B)$ is controllable, then the problem has a solution for any pair $(x_0, x_1)\in \bF^n\times \bF^n$. Moreover, the control law $u_T(t)$ that minimizes the cost is given by
\begin{equation}\label{eq:optimfinTu}
u_T(t) = - B^\dagger e^{ -A^\dagger t} W_{A, B}(T)^{-1} ( x_0 - e^{- AT} x_1), \quad \mbox{for } t \in [0, T],
\end{equation}
where $W_{A, B}(T)$ is the {\em controllability Gramian}: 
\begin{equation}\label{eq:gramian}
W_{A,B}(T) := \int_0^T e^{- A t} B B^\dagger e^{-A^\dagger t} \rrd t.
\end{equation}
Note that our definition of controllability Gramian  follows~\cite{brockett2015finite}, but other conventions are sometimes used whereby factors $e^{AT}$ and $e^{A^\dagger T}$ are added on the left and right, respectively.  

The corresponding minimal cost, i.e., the solution to~\eqref{eq:minenergyprobfinT}, is 
\begin{equation}\label{eq:mincost}(x_0 - e^{-AT} x_1)^\dagger W_{A, B}(T)^{-1} (x_0 - e^{-AT} x_1).
\end{equation}
In this paper, we are interested in the case where $x_1 = 0$, so the minimal cost is reduced to $x_0^\dagger W_{A, B}(T)^{-1} x_0$. 

The minimum energy control problem is a staple control problem, For the case $T<\infty$, it is entirely solved as mentioned above. However, the  {\em infinite-horizon} case remains surprisingly open. 
We provide a solution in this paper.

We now define the problem formally:

\begin{definition}[Admissible control]
    Let $(A, B)$ be stabilizable. 
    A control law $u: [0,\infty) \to \bF^m$ is {\em admissible} for $\dot x(t) = Ax(t)+Bu(t)$, with $x(0) = x_0$, if $u \in \mathrm{L}^2([0,\infty,),\bF^m)$ and if the solution $x(t)$ of the system satisfies $\lim_{t \to \infty} x(t)=0$. We denote the set of admissible controls by $\cU(x_0)$.
\end{definition}

The infinite-horizon minimum energy control problem is then given by:

\begin{problem}[Infinite-horizon minimum energy control]\label{prob:infinitehorizon}
Given $x_0\in \bF^n$, the infinite-horizon {\em minimum energy control problem} for system~\eqref{eq:linsys}, with $(A, B)$ stabilizable, is 
\begin{equation}\label{eq:minenergyprob}
    \min_{u \in \cU(x_0)} \eta(u):=\int_0^\infty \|u(t)\|^2 \rrd t \quad \mbox{s.t. } x(0) = x_0. 
\end{equation}
\end{problem}

The main theorem of the paper states a necessary and sufficient condition on the existence of solution to Problem~\ref{prob:infinitehorizon} for all $x_0\in \bF^n$.

\begin{theorem}\label{th:main0}
Let $\bF$ be either $\R$ or $\C$. 
Let $(A,B)\in \bF^{n\times n}\times \bF^{n\times m}$ be a stabilizable pair. The infinite-horizon minimum energy control problem admits a  solution for all $x_0 \in \bF^n$ if and only if $A$ has no  imaginary eigenvalue. 
\end{theorem}

\mab{Extension to general quadratic cost:} 
Theorem~\ref{th:main0} still holds if we allow the cost in  Problem~\ref{prob:infinitehorizon} to be of the more general form $\int_0^\infty u(t)^\dagger R u(t) \rrd t$, for some positive definite matrix~$R$. To wit, given a triplet $(A,B, R)$ as parameters  for the optimization problem, with $(A, B)$ stabilizable and $R> 0$, 
we consider the  linear system  $\dot y(t) = A y(t) + BR^{-1/2} v(t)$. Clearly, the controllability and stabilizability properties of $(A,B)$ and $(A,BR^{-1/2})$ are the same. 
Also, it is not hard to see that the trajectory $x(t)$ generated by the control input $u(t)$ for system $(A, B)$ is identical to the trajectory $y(t)$ generated by the control $v(t):= R^{1/2} u(t)$ for system $(A, BR^{-1/2})$, as long as both trajectories have the same initial state. In particular, $x(t) \to 0$ if and only if $y(t) \to 0$. 
Further, note that $\|v(t)\| \leq \|R^{1/2}\| \|u(t)\|$ and, conversely, $\|u(t)\| \leq \|R^{-1/2}\| \|v(t)\|$ (the induced $2$-norm  is used here). Together, these items show that if $u$ is admissible for the infinite-horizon problem with parameter $(A,B,R)$, then $v(t)=R^{1/2}u(t)$ is admissible for the problem with parameter $(A,BR^{-1/2},I)$, which is exactly Problem~\ref{prob:infinitehorizon} for the pair $(A,BR^{-1/2})$. Moreover, if $u^*$ is the minimizer of the former, then $v^*$ is the minimizer of the latter, and vice versa.

\mab{The issue with imaginary eigenvalues.}
Problem~\ref{prob:infinitehorizon} is relatively simple if $A$ has no imaginary eigenvalues. In that case, if the pair $(A,B)$ is stabilizable, then there exists a unique optimal solution; we state a precise result in Proposition~\ref{prop:completionsquare} (with the notation of that Proposition, $\bV_{\hspace{-.04cm}\mathsf{a}} = \C^n$ when $A$ has no imaginary eigenvalues).  Moreover, the optimal control takes the familiar {\em feedback form} $u^*(t)= -B^\dagger K x(t)$, for a uniquely defined $K$ which is a solution to some Riccati equation. Implementing the optimal control leads to the closed-loop system $\dot x(t) = (A - BB^\dagger K) x(t)$.  

We will show that  the eigenvalues of $A$ and the eigenvalues of $(A-BB^\dagger K)$ are related in the following way: 
\begin{equation}
\mbox{if } a+\mathrm{i}b \in \spec(A), \mbox{ then}-|a|+\mathrm{i}b \in \spec(A-BB^\dagger K). \tag{$\star$}
\end{equation}
In terms of pole placement, the optimal control is a feedback control law that moves the poles of the original system that lie on the right half of the complex plane to their  symmetric locations about the imaginary axis. This remarkable behavior is shown in Proposition~\ref{prop:theta}.

The above  sheds some light on why the situation when $A$ has imaginary eigenvalues is more complex. When $A$ has no imaginary eigenvalue, the feedback law $u=-B^\dagger Kx$ solves Problem~\ref{prob:infinitehorizon} for {\em all} initial states $x_0 \in \bF^n$. However, if $A$ has imaginary eigenvalue,  by~$(\star)$, there exists initial states (e.g., initial states that are right-eigenvectors of $(A-BB^\dagger K)$ corresponding to imaginary eigenvalues) for which the above feedback control law is not even admissible. 
One is thus forced to either seek an optimal control law of other form or show that such optimal control law does not exist. We will show that the latter statement holds.

% Roughly speaking, imaginary eigenvalues in $A$ can produce trajectories of the system that diverge slowly (here, this means less than exponentially fast), and these trajectories are the ones that require controls that converge very fast two zero to be handled.

\mab{Outline of our approach:} The remaining sections of the paper are devoted to the proof of Theorem~\ref{th:main0}. 
The proof relies on two major results, which we state in Section~\ref{sec:mainresults}. For ease of exposition, we assume in this outline that $A$ is in its Jordan normal form. The first result (Theorem~\ref{thm:main} below) deals with the asymptotic behavior of the inverse of the controllability Gramian $W_{A, B}(T)$, for $(A, B)$ a controllable pair; specifically, we show that $W_{A, B}(T)^{-1}$ converges to the aforementioned matrix $K$, precisely defined in~\eqref{eq:limitK}, as $T\to\infty$. 
Its proof is quite technical and we come back to it below. The second result (Theorem~\ref{thm:main2p}) provides a necessary and sufficient condition on the subspace, denoted by $\bV_{\hspace{-.04cm}\mathsf{a}}$, of initial states $x_0$ for which a (unique) solution  to Problem~\ref{prob:infinitehorizon} exists. 
It turns out, as alluded in the previous paragraph, that this subspace is the entire state space if and only if $A$ has no imaginary eigenvalue. 
Note that Theorem~\ref{thm:main2p} is stated for $\bF=\C$. We state and prove the case of $\bF=\R$ as a corollary (see Corollary~\ref{cor:real}).

After stating these theorems, we present in Section~\ref{sec:proofthm} the proof of Theorem~\ref{thm:main2p}.  
In the proof of the sufficiency part, we assume that $x_0\in \bV_{\hspace{-.04cm}\mathsf{a}}$ and build the optimal control law explicitly, using a classical completion of square argument (see Subsection~\ref{ssec:s2thmm2p}). 
Proving the necessity of those conditions is more subtle;  in a nutshell, our proof goes by showing that the following two statements are true for $x_0\not\in\bV_{\hspace{-.04cm}\mathsf{a}}$: {$(i)$} The minimum value of Problem~\ref{prob:infinitehorizon}, if it exists, has to be $x_0^\dagger K x_0$; $(ii)$ There is no {\em admissible} control law that can attain such value of the cost function. Together, these statements preclude the existence of a solution for these initial states.
\begin{figure}[]
\centering
\begin{subfigure}[t]{0.6\textwidth}
        \includegraphics{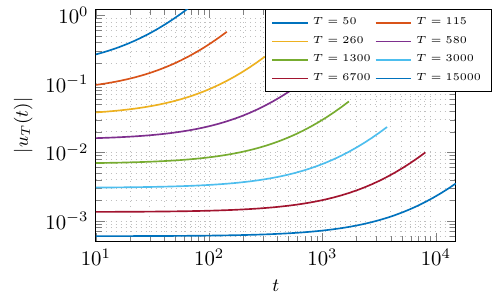}        \caption{}
    \end{subfigure}~\begin{subfigure}[t]{0.4\textwidth}
\includegraphics{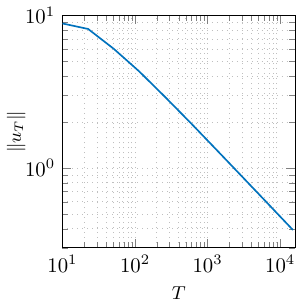}
       \caption{}
    \end{subfigure}%
\caption{For the controllable matrix pair $(A,B)$ with $A$ a $3$-by-$3$ Jordan block with zero eigenvalues, and $B=[0;0;1]$, we plot in (a) the absolute value of the optimal control $u_T(t)$ in~\eqref{eq:optimfinTu} which drives the system from $x_0=[1;1;1]$ to $x(T)=0$ for increasing $T$, and in (b) the $\mathrm{L}^2$ norm of $u_T$. We see that as $T$ grows, $\|u_T\|$ become smaller and, in the limit, convergence to $0$. However, since the system is not stabilizable, a zero control does not drive $x_0$ to the origin.  This illustrates the issue with imaginary eigenvalues.}
\end{figure}

The remainder of the paper, namely Sections~\ref{sec:pfimaginary},~\ref{sec:pfnonpositive}, and~\ref{sec:pfgeneral}, are devoted to the proof of Theorem~\ref{thm:main}. The proof is technical and requires us to track the asymptotic convergence rates of the submatrices of $W_{A, B}(T)^{-1}$ --- these submatrices are the ones  that arise from the decomposition of the spectrum of $A$ into eigenvalues with positive, negative, and zero real parts. 

As mentioned in~\eqref{eq:mincost}, for $(A, B)$ a controllable pair, $x_0^\dagger W_{A, B}(T)^{-1} x_0$ is the solution to the finite-horizon minium energy control problem.  
It is intuitive that the corresponding finite-horizon optimal control law for initial states in the stable subspace of $A$ 
converge to zero as $T\to\infty$, since the uncontrolled dynamics of the system flows such states to the origin. The corresponding principal submatrix of $W_{A,B}(T)^{-1}$ is shown to vanish asymptotically, coinciding with the expected behavior. 
Reciprocally, initial states in the unstable and in the center subspace of the system always require some control effort. While the  principal submatrix of $W_{A,B}(T)^{-1}$ corresponding to the unstable subspace converges to an invertible matrix (which we precisely characterize below), rather surprisingly, the principal submatrix corresponding to the center subspace converges to $0$, albeit at the slow rate of $1/T$.
This is a key result we will prove---and it is at the root of the non-existence of optimal controls in this case.

We approach the proof by first establishing Theorem~\ref{thm:main} for the special case where $A$ has only imaginary eigenvalues, and then using {\em nested} Schur complements to prove the theorem for the general case. The proof for the imaginary case is presented in Section~\ref{sec:pfimaginary}. Evaluating the Schur complement involves evaluating the asymptotic product of several submatrices of $W_{A, B}(T)$. However, we will see that these submatrices do not each converge on their own, and hence showing that their {\em product} converges is a delicate task. We approach it by introducing what we call below ``buffer terms'' (see Subsection~\ref{ssec:buff}) to show that some specific groupings of terms converge. Using this approach, together with the result established for the imaginary case, we prove Theorem~\ref{thm:main} first for the case where the eigenvalues of $A$ have non-positive real parts (Section~\ref{sec:pfnonpositive}) and then, for the general case (Section~\ref{sec:pfgeneral}).

\mab{Notations:} 
We denote by $\bfo$ the vector of all ones, whose dimension will be clear from context.
Given a complex matrix $Z$, we let $\bar Z$ be its complex conjugate and $Z^\dagger$ be its Hermitian conjugate. 
Given a vector space $\bV$ over $\bF$, we let $\dim_{\bF}\bV$ be its dimension.
Given a square matrix $A$, we denote by $\spec(A)$ its  set of eigenvalues, allowing for repeated eigenvalues.  Given a function $f:[0,\infty) \to \C$, we denote its $\mathrm{L}^2$-norm as $\|f\|_2 = (\int_0^\infty f^2(t)\rrd t)^{1/2}$. 
We say that a matrix pair $(A, B)\in \bF^{n\times n}\times \bF^{n\times m}$ is {\em controllable} if the rank of the controllability matrix
$$C(A, B) := 
\begin{bmatrix}
B & AB & \cdots & A^{n-1}B
\end{bmatrix}
$$
is full (i.e., $n$). 
Note in particular that if $A = \diag(A_1, A_2)$ is block diagonal and if we decompose $B = [B_1; B_2]$ so that the dimensions of $B_i$ match those of $A_i$, then $(A_i,B_i)$ is controllable for $i = 1,2$. The uncontrollable  of $(A,B)$ is defined as 
Further, we say that $(A, B)$ is stabilizable if there exists a matrix $F\in \bF^{m\times n}$ such that $(A + BF)$ is Hurwitz. It is well-known that $A$ is stabilizable if and only if the subspace spanned by the generalized eigenvectors of $A$ corresponding to eigenvalues with nonnegative real parts is included in the range of $C(A,B)$.

\section{Main Results}\label{sec:mainresults}

We  prove a stronger statement than Theorem~\ref{th:main0},  in that we  characterize the initial conditions for which the infinite horizon minimum energy control problem has a solution. We will state the result first for the case $\bF=\C$, and then prove in Corollary~\ref{cor:real} that the result also holds for the case $\bF=\R$.

From now on, we consider the complex case unless specified.  We start with some preliminaries.
Let $P\in \C^{n\times n}$ be a nonsingular matrix such that $P^{-1}AP$ is in the Jordan normal form~\cite{luenberger1967canonical}:
\begin{equation}\label{eq:PAP}
P^{-1}A P = 
\begin{bmatrix}
J_{\mathsf{u}} & 0 \\
0 & J_{\mathsf{s}}
\end{bmatrix} =:J,
\end{equation}
where the eigenvalues of $J_{\mathsf{u}}\in \C^{n_{\mathsf{u}}\times n_{\mathsf{u}}}$ (resp., $J_{\mathsf{s}}\in \C^{n_{\mathsf{s}}\times n_{\mathsf{s}}}$) have positive (resp. non-positive) real parts. Let $C:= P^{-1}B$ and we decompose 
\begin{equation}\label{eq:defC}
C = 
\begin{bmatrix}
C_{\mathsf{u}} \\
C_{\mathsf{s}} 
\end{bmatrix},
\end{equation}
where $C_{\mathsf{u}}\in \C^{n_{\mathsf{u}}\times m}$ and $C_{\mathsf{s}} \in \C^{n_{\mathsf{s}}\times m}$. For a later purpose, we further decompose 
\begin{equation}\label{eq:decompJ2C2}
J_{\mathsf{s}}=
\begin{bmatrix}
J_{\mathsf{o}} & 0 \\
0 & J_{\mathsf{a}}
\end{bmatrix} \quad \mbox{and} \quad 
C_{\mathsf{s}} = 
\begin{bmatrix}
C_{\mathsf{o}} \\
C_{\mathsf{a}}
\end{bmatrix}, 
\end{equation} 
where $J_{\mathsf{o}}\in \C^{n_{\mathsf{o}}\times n_{\mathsf{o}}}$ has imaginary eigenvalues and $J_{\mathsf{a}}\in \C^{n_{\mathsf{a}}\times n_{\mathsf{a}}}$ has eigenvalues with negative real parts; note that $n_{\mathsf{o}}+n_{\mathsf{a}}=n_{\mathsf{s}}$. We correspondingly decompose $C_{\mathsf{s}}$ as above. 
To summarize, the real parts of the eigenvalues of
\begin{itemize}
    \item $J_{\mathsf{u}}$ are positive;
    \item $J_{\mathsf{s}}$ are less than or equal to zero;
    \item $J_{\mathsf{o}}$ are zero;
    \item $J_{\mathsf{a}}$ are negative.
\end{itemize}

Let $(A, B)$ be a stabilizable pair. Then, $(J,C)$ is stabilizable as well and its controllable subspace contains the invariant subspace spanned by the generalized right eigenvectors of $J_{\mathsf{u}}$.  This, in particular, implies that $(J_{\mathsf{u}}, C_{\mathsf{u}})$ is controllable.  
Let 
\begin{equation}\label{eq:defW1}
W_{\hspace{-.04cm}\mathsf{u}}:= W_{J_{\mathsf{u}}, C_{\mathsf{u}}}(\infty) = \int_0^\infty e^{-J_{\mathsf{u}} t} C_{\mathsf{u}} C_\mathsf{u}^\dagger e^{-J_\mathsf{u}^\dagger t} \rrd t\in \C^{n_{\mathsf{u}}\times n_{\mathsf{u}}}, 
\end{equation} 
which is also the unique positive semi-definite solution to the following Lyapunov equation:
\begin{equation}\label{eq:lyapunovW1}
J_{\mathsf{u}} W_{\hspace{-.04cm}\mathsf{u}} + W_{\hspace{-.04cm}\mathsf{u}} J_{\mathsf{u}}^\dagger  - C_{\mathsf{u}}  C_{\mathsf{u}}^\dagger = 0.
\end{equation}
We further let
\begin{equation}\label{eq:limitK}
K:= (P^{-1})^\dagger
\begin{bmatrix}
W_{\hspace{-.04cm}\mathsf{u}}^{-1} & 0 \\
0 & 0 
\end{bmatrix} P^{-1} \in \C^{n \times n}.
\end{equation}
Note that if $n_{\mathsf{u}} = 0$, then $K = 0$. 
From~\eqref{eq:PAP} and~\eqref{eq:lyapunovW1}, we have that $K$ satisfies the Riccati equation:
\begin{equation}\label{eq:riccatiK}
A^\dagger K + K A - K BB^\dagger K = 0.
\end{equation}
The main technical result of this paper, on which the proof of Theorem~\ref{th:main0} is built, is the following:

\begin{theorem}\label{thm:main}
    Let $(A,B)\in \C^{n\times n} \times \C^{n\times m}$ be a controllable pair, and let $W_{A, B}(T)$ and $K$ be defined as in~\eqref{eq:gramian} and~\eqref{eq:limitK}, respectively. Then, 
    $$
    \lim_{T\to\infty} W_{A, B}(T)^{-1} = K.
    $$
\end{theorem}
 
\begin{figure}[h]
    \includegraphics{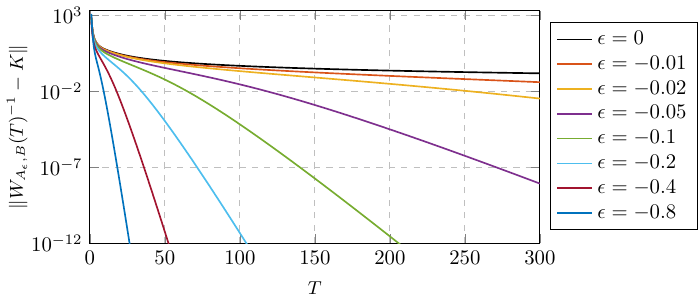}
    \caption{Illustration of the convergence of the inverse of the controllability Gramian  for different matrix pairs $(A_\epsilon, B)$.}\label{fig:fig_w}
\end{figure}
To illustrate Theorem~\ref{thm:main}, we computed numerically the norm of the  difference between $W_{A,B}(T)$ and its limit $K$ as defined in  the Theorem; the results are shown in Figure~\ref{fig:fig_w}. We did so for several matrix pairs $(A_\epsilon, B)$, with $A_\epsilon$ taking eigenvalues $3,2,-1$ and $\epsilon$, and observe that, as mentioned earlier, the convergence of $W_{A_\epsilon, B}(T)^{-1}$ gets very slow as $\epsilon \to 0$. 
Precisely, we consider a parameterized family of $(A_\epsilon, B)$ pairs: 
\begin{equation*}\label{eq:abpair}
A_\epsilon = 
\begin{bmatrix}
3 & 0 & 0 & 0 & 0 & 0 \\
0 & 2 & 0 & 0 & 0 & 0 \\ 
0 & 0 & \epsilon & 1 & 0 & 0 \\
0 & 0 & 0  & \epsilon & 1 & 0 \\
0 & 0 & 0 & 0 & \epsilon & 0 \\
0 & 0 & 0 & 0 & 0 & -1
\end{bmatrix} \quad \mbox{and} \quad 
B=
\begin{bmatrix}
0 & 1 \\
1 & 0 \\
0 & 0 \\
1 & 0 \\
0 & 1 \\
0 & 1
\end{bmatrix}.
\end{equation*}

We now introduce the following operator: given a complex number $\lambda = a+ \mathrm{i} b$, with $a,b \in \R$, we let $$\theta: \C \to \C: a+\mathrm{i}b \mapsto -|a|+\mathrm{i}b.$$
Note that $\theta(\lambda)$ always has a nonpositive real part.

Let $\bV_{\hspace{-.04cm}\mathsf{u}}$, $\bV_{\hspace{-.04cm}\mathsf{o}}$, and $\bV_{\hspace{-.04cm}\mathsf{a}}$  be the subspaces over $\C$ spanned by the generalized right eigenvectors of $(A-BB^\dagger K)$ corresponding to the eigenvalues with positive, zero, and negative real parts respectively. 

The following proposition, which relates the subspaces introduced above to the $\theta$ operator, is needed to state the next main result: 

\begin{proposition}\label{prop:theta}
    Let $(A, B)\in \C^{n\times n} \times \C^{n\times m}$ be a stabilizable pair, and $K$ be given as in~\eqref{eq:limitK}.  
    Then, 
    \begin{equation}\label{eq:abbkthetaa}
        \spec(A - BB^\dagger K) = \theta(\spec(A)).
        \end{equation} 
        In particular, 
        \begin{equation}\label{eq:tr1tr21tr22}
        \dim_\C \bV_{\hspace{-.04cm}\mathsf{u}} = 0,  \quad 
        \dim_\C \bV_{\hspace{-.04cm}\mathsf{o}} = n_{\mathsf{o}}, \quad \mbox{and} \quad
        \dim_\C \bV_{\hspace{-.04cm}\mathsf{a}} = n_{\mathsf{u}} + n_{\mathsf{a}}.
        \end{equation}
\end{proposition}
\begin{proof}
     
We define
\begin{equation}\label{eq:deftildeK}
\widetilde K := P^\dagger K P = 
\begin{bmatrix}
W_{\hspace{-.04cm}\mathsf{u}}^{-1} & 0 \\
0 & 0
\end{bmatrix}. 
\end{equation}
Clearly, $\spec(P^{-1}(A-BB^\dagger K)P) = \spec(A-BB^\dagger K)$ and
\begin{equation*}
P^{-1}(A - BB^\dagger K) P = J - C C^\dagger  \widetilde K = 
\begin{bmatrix}
J_{\mathsf{u}} - C_{\mathsf{u}} C_{\mathsf{u}}^\dagger W_{\hspace{-.04cm}\mathsf{u}}^{-1} & 0 \\
-C_{\mathsf{s}} C_{\mathsf{u}}^\dagger W_{\hspace{-.04cm}\mathsf{u}}^{-1} & J_{\mathsf{s}}
\end{bmatrix}.
\end{equation*}
By construction, the eigenvalues of $J_{\mathsf{s}}$ have non-positive real parts, and are thus mapped to themselves by $\theta$. Hence,~\eqref{eq:abbkthetaa} follows if we  show that 
\begin{equation}\label{eq:specLambda1}
\spec(J_{\mathsf{u}} - C_{\mathsf{u}} C_{\mathsf{u}}^\dagger W_{\hspace{-.04cm}\mathsf{u}}^{-1}) = \theta(\spec(J_{\mathsf{u}})).
\end{equation}
To establish~\eqref{eq:specLambda1}, 
we multiply~\eqref{eq:lyapunovW1} by $W_{\hspace{-.04cm}\mathsf{u}}^{-1}$ on the right (recall that $W_{\hspace{-.04cm}\mathsf{u}}$ is nonsingular), we obtain that
$$
W_{\hspace{-.04cm}\mathsf{u}}J_{\mathsf{u}}^\dagger W_{\hspace{-.04cm}\mathsf{u}}^{-1} = -(J_{\mathsf{u}} - C_{\mathsf{u}} C_{\mathsf{u}}^\dagger W_{\hspace{-.04cm}\mathsf{u}}^{-1}),  
$$
and hence,  
$$\spec(J_{\mathsf{u}}^\dagger)= \spec(W_{\hspace{-.04cm}\mathsf{u}}J_{\mathsf{u}}^\dagger W_{\hspace{-.04cm}\mathsf{u}}^{-1}) = - \spec(J_{\mathsf{u}} - C_{\mathsf{u}}C_{\mathsf{u}}^\dagger W_{\hspace{-.04cm}\mathsf{u}}^{-1}).$$
Note that if  $\lambda$ has a {\em positive} real part,  then  $\theta(\lambda)= - \bar \lambda$. Now  let  $\lambda \in \spec(J_{\mathsf{u}})$ (thus $\lambda$ has a positive real part), then  $\bar \lambda \in \spec(J_{\mathsf{u}}^\dagger)$ and, by the above relation, $-\bar \lambda \in \spec(J_{\mathsf{u}} - C_{\mathsf{u}}C_{\mathsf{u}}^\dagger W_{\hspace{-.04cm}\mathsf{u}}^{-1})$,
which proves~\eqref{eq:specLambda1}. The second part of the proposition i.e.,~\eqref{eq:tr1tr21tr22} follows directly from~\eqref{eq:abbkthetaa}.
\end{proof}

Using Theorem~\ref{thm:main}, we will prove the following result, which is essentially Theorem~\ref{th:main0} for $\bF=\C$:

\begin{theorem}\label{thm:main2p}
    Let $(A, B) \in \C^{n \times n} \times \C^{n \times m}$ be a stabilizable pair, and $K$ be given as in~\eqref{eq:limitK}.  
    Then, Problem~\ref{prob:infinitehorizon} with $\bF=\C$ has a solution if and only if  $x_0\in \bV_{\hspace{-.04cm}\mathsf{a}}$ and it is then given by  $x_0^\dagger K x_0$. Moreover, there exists a unique control law $u^*\in \cU(x_0)$ that minimizes the cost and it is given by $u^*(t) := - B^\dagger K x(t)$ for $t\geq 0$. 
\end{theorem}

Theorem~\ref{thm:main2p} also holds in the real case, i.e., for  real system matrices $A, B$, with real state $x(t)$ and real control $u(t)$. We establish this now.

First, note that if the pair $(A, B)$ is real, then $W_{A, B}(T)$ is real for all $T> 0$. Thus, by Theorem~\ref{thm:main}, $K$ is real, and so is $(A - BB^\top K)$.
It follows that the complex vector spaces $\bV_{\hspace{-.04cm}\mathsf{o}}$ and $\bV_{\hspace{-.04cm}\mathsf{a}}$ are closed under conjugation, i.e., if $v \in \bV_{\hspace{-.04cm}*}$, then $\bar v \in \bV_{\hspace{-.04cm}*}$ for $* = \mathsf{o},\mathsf{a}$. 

Second, note that any complex vector space $\bV \subseteq \C^n$ of dimension~$k$ that is closed under conjugation admits a basis of {\em real} vectors. Indeed, if $\{v_1 + \mathrm{i} w_{1},\ldots, v_k + \mathrm{i} w_k\}$, with $v_j,w_j \in \R^n$, is such a basis, then so is $\{v_1 - \mathrm{i} w_1, \ldots, v_k - \mathrm{i} w_k\}$. Thus, $v_j, w_j \in \bV$ and, moreover, their span over $\C$ is $\bV$.

The above two items together show that $\bV_{\hspace{-.04cm}*}$, for $* = \mathsf{o},\mathsf{a}$, admits a basis of {\em real} vectors. 
We then let $\bR_{*}$ be the vector space, over $\R$, spanned by these vectors. 
In particular, by Proposition~\ref{prop:theta}, we have that $\dim_{\R} \bR_{\mathsf{o}} = n_{\mathsf{o}}$ and $\dim_{\R} \bR_\mathsf{a} = n_{\mathsf{u}} + n_{\mathsf{a}}$. 

We now state the counterpart of Theorem~\ref{thm:main2p} for the real case:

\begin{corollary}[The real case]\label{cor:real}
     Let $(A, B) \in \R^{n\times n} \times \R^{n \times m}$ be a stabilizable pair, and $K\in \R^{n\times n}$ be given as in~\eqref{eq:limitK}.  
    Then, Problem~\ref{prob:infinitehorizon} with $\bF = \R$ has a solution if and only if  $x_0\in \bR_\mathsf{a}$ and it is then given by  $x_0^\top K x_0$. Moreover, there exists a unique control law $u^*\in \cU(x_0)$ that minimizes the cost and it is given by $u^*(t) = - B^\top K x(t)$ for $t\geq 0$. 
\end{corollary}

\begin{proof}
    We will denote the set of admissible {\em real-} and {\em complex-}valued control laws by $\cU_\R(x_0)$ and $\cU_\C(x_0)$, respectively. 
    
    We assume that $x_0\in \bR_\mathsf{a}$ and show that Problem~\ref{prob:infinitehorizon} admits the solution $\eta(u^*) = x_0^\top K x_0$, with $u^* \in \cU_\R(x_0)$ the unique minimizer. 
    First, since $K$ and $B$ are real, $u^* \in \cU_\R(x_0)$. 
    Next, since $\bR_\mathsf{a}$ is contained in $\bV_{\hspace{-.04cm}\mathsf{a}}$ by the arguments above Corollary~\ref{cor:real}, we have that $x_0\in \bV_{\hspace{-.04cm}\mathsf{a}}$. 
    Hence, we can apply Theorem~\ref{thm:main2p} to system~\eqref{eq:linsys} with initial condition $x_0$, and it yields that  $u^*(t) = -B^\top Kx(t)$  is the unique minimizer for Problem~\ref{prob:infinitehorizon}, but with $\bF=\C$, and the cost is $\eta(u^*)=x_0^\top K x_0$.     
    Because an admissible solution for Problem~\ref{prob:infinitehorizon} with $\bF=\R$ is also admissible with $\bF=\C$, and because the (unique) minimizer $u^*$ of the latter problem is an element of $\cU_{\R}(x_0)$, we have that $u^*$ is the unique minimizer for the former problem.

    We now assume that $x_0\not\in \bR_\mathsf{a}$ and show that Problem~\ref{prob:infinitehorizon} with $\bF = \R$ does not admit a solution. 
    Proceeding by contradiction, assume that it admits a minimizer $v^* \in \cU_\R(x_0)$. First, note that $x_0 \notin \bV_{\hspace{-.04cm}\mathsf{a}}$; indeed, since $\bV_{\hspace{-.04cm}\mathsf{a}}$ admits a basis of real vectors and since $x_0$ is real, if $x_0 \in \bV_{\hspace{-.04cm}\mathsf{a}}$, then it can be expressed as a linear combination of these real vectors with necessarily real coefficients, which then implies that $x_0\in \bR_\mathsf{a}$.  
    Next,  since $x_0 \notin \bV_{\hspace{-.04cm}\mathsf{a}}$, Theorem~\ref{thm:main2p} implies that $v^*$ is not a minimizer of Problem~\ref{prob:infinitehorizon} for $\bF = \C$. Hence, there exists  control law $v'\in \cU_{\C}(x_0)$ such that 
    \begin{equation}\label{eq:etav'}
    \eta(v') < \eta(v^*).
    \end{equation}
    We write $v'(t) = v'_1(t) + \mathrm{i}v'_2(t)$, where $v'_1(t), v'_2(t) \in \R^m$. We similarly decompose the solution of $\dot x(t) = Ax(t) +Bv'(t)$ as $x(t)=x_1(t) + \mathrm{i}x_2(t)$ with $x_i(t) \in \R^n$. Since $A$ and $B$ are real, the dynamics of $x_1(t)$ and $x_2(t)$ are completely decoupled, i.e.,
    \begin{equation}\label{eq:splitode}
    \dot x_i(t) = Ax_i(t) +B v'_i(t), \quad \mbox{ with } x_{0,1}=x_0 \mbox{ and } x_{0,2}=0.
    \end{equation}
    Since $v'$ is such that $\lim_{t \to \infty} x(t)=0$, the solution of~\eqref{eq:splitode} with $i=1$ satisfies $\lim_{t \to \infty} x_1(t)=0$. Moreover, 
    \begin{equation}\label{eq:splitcost}
    \int_0^\infty \|v'_1(t)\|^2 \rrd t \leq \int_0^\infty \|v'(t)\|^2 \rrd t < \infty.
    \end{equation} 
    The above arguments show that $v'_1 \in \cU_\R(x_0)$. But then,~\eqref{eq:etav'} and~\eqref{eq:splitcost} together imply that $\eta(v_1') < \eta(v^*)$, which   contradicts the optimality of $v^*$ and proves the result.
\end{proof}

We now prove Theorems~\ref{thm:main} and~\ref{thm:main2p}. The proof of the first theorem is rather technical, and we delay it until after the proof of Theorem~\ref{thm:main2p}.

\section{Proof of Theorem~\ref{thm:main2p}}\label{sec:proofthm}
We prove the necessity and sufficiency in Subsections~\ref{ssec:s2thmm2p} and~\ref{ssec:necessityTH}, respectively. The sufficiency is relatively straightforward, and is based on a completion of square argument. The proof of necessity is more delicate and will be carried out by contradiction. We first show that if Problem~\ref{prob:infinitehorizon} admits a solution, then it {\em necessarily equals}  $\eta(u^*)$ with $u^* = -B^\dagger K x(t)$ the unique minimizer.  
This is done by introducing two auxiliary optimal control problems (see~\eqref{eq:minenergyprobunstable} and~\eqref{eq:defauxiloptim}) that will provide comparisons with Problem~\ref{prob:infinitehorizon}. 
Then, the optimal control yields  the closed-loop system  $\dot x(t) = (A-BB^\dagger K) x(t)$. 
But now,  recall that $\bV_{\hspace{-.04cm}\mathsf{a}}$ is spanned by the generalized right eigenvectors of $(A-BB^\dagger K)$ corresponding to eigenvalues with negative real parts---by Proposition~\ref{prop:theta}, this subspace equals $\C^n$ {\em unless} $A$ has imaginary eigenvalues. 
Thus, if $x(0) = x_0\not\in \bV_{\hspace{-.04cm}\mathsf{a}}$, then $x(t)$ will not converge to $0$, which implies that $u^*(t)$ is not admissible.

\subsection{Proof of sufficiency}\label{ssec:s2thmm2p}

In this subsection, we establish the following proposition:

\begin{proposition}\label{prop:completionsquare}
    Suppose that $x_0\in \bV_{\hspace{-.04cm}\mathsf{a}}$; then, Problem~\ref{prob:infinitehorizon} has the solution $\eta(u^*) = x_0^\dagger K x_0$, with $u^*(t) = - B^\dagger K x(t)$.
\end{proposition}

\begin{proof}
The proof  follows a classical completion of square argument. Let $u \in \cU(x_0)$ be an admissible control, and $x(t)$ be the solution of the linear system generated by this control input $u(t)$, with $x_0$ the initial condition. 
Then, using~\eqref{eq:riccatiK}, the cost incurred by $u(t)$ is 
\begin{align}
    \int_0^\infty \|u(t)\|^2 \rrd t & = \int_0^\infty  \left[\|u(t)\|^2 - x(t)^\dagger(A^\dagger K + K A - K BB^\dagger K)x(t)\right] \rrd t \notag \\
    & = \int_0^\infty \left [ \|u(t) + B^\dagger K x(t)\|^2 - x(t)^\dagger K (A x(t) + Bu(t))\right. \notag\\
    & \hspace{4.3cm} \left. - (A x(t) + B u(t))^\dagger K x(t) \right ] \rrd t\notag\\
    &=  \int_0^\infty \left [ \|u(t) + B^\dagger K x(t)\|^2 - x(t)^\dagger K \dot x(t) - \dot x(t)^\dagger K x(t) \right ] \rrd t\notag\\
    &=  \int_0^\infty \|u(t) + B^\dagger K x(t)\|^2 \rrd t - \int_0^\infty \frac{\rrd}{\rrd t} (x(t)^\dagger K x(t)) \rrd t\notag\\
    &= \label{eq:ineqopt} \int_0^\infty \|u(t) + B^\dagger K x(t)\|^2 \rrd t + x_0^\dagger K x_0 \geq x_0^\dagger K x_0,
\end{align}
where we used the fact that $\lim_{t \to \infty}x(t)=0$ for the last equality.
It should be clear from the above that the minimal value of  Problem~\ref{prob:infinitehorizon} is bounded below by $x_0^\dagger K x_0$. This lower bound can be reached only by the control law $u^*(t) = -B^\dagger K x(t)$, for all $t \geq 0$. 

By Proposition~\ref{prop:theta},  this feedback control law yields an {\em exponentially} stable closed loop system $\dot x(t) = (A - BB^\dagger K) x(t)$ when restricted to the invariant subspace $ \bV_{\hspace{-.04cm}\mathsf{a}}$. 
Thus, by the hypothesis on $x_0$, we have that the solution $x^*(t)$ of~\eqref{eq:linsys} generated by $u^*(t)$ decays to zero exponentially fast, which implies that $u^*(t) \in \cU(x_0)$. 
This completes the proof of the sufficiency of Theorem~\ref{thm:main2p}. 
\end{proof}

The arguments in the above proof also lead to the following fact:

\begin{corollary}\label{cor:noimagA}
    Suppose that $A$ does not have any  imaginary eigenvalue and that $(A, B)$ is stabilizable; then, there is a unique Hermitian solution $K$ to the Riccati equation~\eqref{eq:riccatiK} such that $(A - BB^\dagger K)$ is Hurwitz, and it is given by~\eqref{eq:limitK}.
\end{corollary} 

\begin{proof}
    First, note that if $A$ has no  imaginary eigenvalue, then by Proposition~\ref{prop:theta}, 
    $\bV_{\hspace{-.04cm}\mathsf{a}} =\C^n$. 
    Assume that there exists another Hermitian solution $K'$ to~\eqref{eq:riccatiK} such that $(A - BB^\dagger K')$ is Hurwitz. Let $u'(t):= -B^\dagger K' x'(t)$, where $x'(t)$ is the solution to $$\dot x'(t) = Ax'(t)+Bu'(t) = (A - BB^\dagger K')x'(t), \quad \mbox{with } x'(0) = x_0.$$ 
    Then, by the same arguments in the proof of Proposition~\ref{prop:completionsquare}, we have that $u'(t)$ minimizes the cost $\eta(u)$ and the minimal cost is given by 
    $\eta(u') = x_0^\dagger K' x_0$. 
    But then, from~\eqref{eq:ineqopt}, we have $x_0^\dagger K' x_0 = x_0^\dagger K x_0$ which holds for all $x_0\in \C^n$. We  thus conclude that $K' = K$.
\end{proof}

\subsection{Proof  of  necessity}\label{ssec:necessityTH}
We show here that if $x_0\notin \bV_{\hspace{-.04cm}\mathsf{a}}$, then Problem~\ref{prob:infinitehorizon} does not admit a solution. 
Given $x(t)\in \C^n$, we let $$y(t) := P^{-1} x(t) = \begin{bmatrix} y_{\mathsf{u}}(t) \\ y_{\mathsf{s}}(t) \end{bmatrix},$$ 
where $y_{\mathsf{u}}(t)\in \C^{n_{\mathsf{u}}}$ and $y_{\mathsf{s}}(t) \in \C^{n_{\mathsf{s}}}$. The dynamics of $y(t)$ is then given by
\begin{equation}\label{eq:tildexdynamics}
\dot {y}(t) = J y(t) + C u(t).
\end{equation}
Consider the following optimal control problem for~\eqref{eq:tildexdynamics}: 
\begin{equation}\label{eq:minenergyprobunstable}
 \min_{u \in \mathrm{L}^2([0,\infty),\C^m)
    } \eta(u)=\int_0^\infty \|u(t)\|^2 \rrd t \quad \mbox{s.t. }  y(0) = P^{-1} x_0 \quad \mbox{and} \quad \lim_{t\to\infty} y_\mathsf{u}(t)= 0,
\end{equation}
which differs from~\eqref{eq:minenergyprob}  in that we require that  {\em only} $y_{\mathsf{u}}(t)$  be asymptotically zero. 
Observe that $u^*(t)$, defined in Theorem~\ref{th:main0}, can  be written as a feedback control in $y(t)$ as
$$
u^*(t) = - B^\dagger K x(t) = - C^\dagger \widetilde K y(t), \quad \mbox{for } t \geq 0,
$$
where $\widetilde K$ is defined in~\eqref{eq:deftildeK}. We have the following result:

\begin{lemma}\label{lem:tildeJK}
The following two items hold:
\begin{enumerate}
    \item The solution to the optimal control problem~\eqref{eq:minenergyprobunstable} is given by 
\begin{equation}\label{eq:minimalcostfortildeJ}
\eta(u^*) = x_0^\dagger K x_0. 
\end{equation}
Moreover, $u^*$ is the unique control law that minimizes the cost.  
\item For any  $u\in \cU(x_0)$, it holds that
$\eta(u) \geq  \eta(u^*)$.
\end{enumerate}
\end{lemma}

\begin{proof}
    Let $y_0 = P^{-1} x_0$, and we partition $y_0 = (y_{0,\mathsf{u}}, y_{0,\mathsf{s}})$ with $y_{0,\mathsf{u}}\in \C^{n_{\mathsf{u}}}$. 
    Since $J$ is block diagonal, the dynamics of $y_{\mathsf{u}}(t)$ do not depend on $y_{\mathsf{s}}(t)$, i.e., 
    \begin{equation}\label{eq:y1indep}
    \dot y_{\mathsf{u}}(t)=J_{\mathsf{u}}y_{\mathsf{u}}(t)+C_{\mathsf{u}}u(t).
    \end{equation}
    The optimization problem~\eqref{eq:minenergyprobunstable} is thus in fact independent from $y_{\mathsf{s}}(t)$ and can be reduced  to the simpler (yet equivalent) problem:
   \begin{equation}\label{eq:minenergyreduced}
 \min_{u \in \mathrm{L}^2([0,\infty), \C^m)
    } \eta(u)=\int_0^\infty \|u(t)\|^2 \rrd t \quad \mbox{s.t. }  y_{\mathsf{u}}(0) = y_{0,\mathsf{u}} \quad \mbox{and} \quad \lim_{t\to\infty} y_{\mathsf{u}}(t)= 0,
\end{equation} 
for the subsystem~\eqref{eq:y1indep}.

Since all the eigenvalues of $J_{\mathsf{u}}$ have positive real parts, 
by the sufficiency of  Theorem~\ref{thm:main2p} proved in Subsection~\ref{ssec:s2thmm2p}, we know that the above problem~\eqref{eq:minenergyreduced} has a solution, which is given by
$$
y_{0,\mathsf{u}}^\dagger W_{J_{\mathsf{u}}, C_{\mathsf{u}}}^{-1}(\infty) y_{0,\mathsf{u}} = y_{0,\mathsf{u}}^\dagger W_{\hspace{-.04cm}\mathsf{u}}^{-1} y_{0,\mathsf{u}} = y_0^\dagger \widetilde K y_0= x_0^\dagger K x_0,
$$
where the first equality follows from the definition of $W_{\hspace{-.04cm}\mathsf{u}}$ (see~\eqref{eq:defW1}) and the second equality follows directly from~\eqref{eq:deftildeK} and the construction of $y_{0,\mathsf{u}}$. 
Moreover, the unique control law that minimizes the cost is $- C_{\mathsf{u}}^\dagger W_{\hspace{-.04cm}\mathsf{u}}^{-1} y_{\mathsf{u}}(t)$, which coincides with $u^*(t)$; indeed,  we have that $$u^*(t) = - C^\dagger \widetilde K y(t) = - C_{\mathsf{u}}^\dagger W_{\hspace{-.04cm}\mathsf{u}}^{-1} y_{\mathsf{u}}(t).
$$ 
This proves the first item of the lemma.

To establish the second item, we note that any admissible $u$ that drives $x(t)$ asymptotically to $0$ also drives $y_{\mathsf{u}}(t)$ to $0$. Thus, 
$\eta(u) \geq \eta(u^*)$.
\end{proof}

We also have the following result:

\begin{lemma}\label{lem:upperboundforJ}
If Problem~\ref{prob:infinitehorizon} has a solution $\eta(\widetilde u^*)$ for some $\widetilde u^*\in \cU(x_0)$, then $\eta(\widetilde u^*) \leq x_0^\dagger K x_0$. 
\end{lemma}

\begin{proof}

We again let $y(t):= P^{-1}x(t)$ and decompose $J = \diag(J_1, J_{\mathsf{a}})$ where   
$$ 
J_1 := 
\begin{bmatrix}
J_{\mathsf{u}}& 0 \\
0 & J_{\mathsf{o}}
\end{bmatrix},
$$
and where we recall that $J_{\mathsf{u}} \in \C^{n_{\mathsf{u}} \times n_{\mathsf{u}}}$ and $J_{\mathsf{o}} \in \C^{n_{\mathsf{o}} \times n_{\mathsf{o}}}$ have eigenvalues with positive and zero real parts, respectively. 
Correspondingly, we decompose $C$ and $y(t)$ as
$$
C = 
\begin{bmatrix}
C_1\\
C_{\mathsf{a}}
\end{bmatrix} \quad \mbox{and} \quad
y(t) = 
\begin{bmatrix}
y_1(t) \\
y_{\mathsf{a}}(t)
\end{bmatrix},
$$
where $C_1\in \C^{(n_{\mathsf{u}}+n_{\mathsf{o}})\times m}$  
and $y_1(t) \in \C^{n_{\mathsf{u}}+n_{\mathsf{o}}}$. Since $(J, C)$ is stabilizable and since all the eigenvalues of $J_{1}$ have non-negative real parts, $(J_{1}, C_{1})$ is controllable. 

We again let $y_0 := P^{-1} x_0$ and decompose $y_0= (y_{0,1}, y_{0,\mathsf{a}})$.   
Now, consider the auxiliary {\em finite-horizon} optimal control problem for the system $\dot y_1(t) = J_{1} y_{1}(t)+C_{1} u(t)$,  
\begin{equation}\label{eq:defauxiloptim}
    \min_{u \in \rL^2([0,T],\C^m)} \eta_T(u)= \int_0^T \|u\|^2 \rrd t \quad \mbox{s.t. } y_1(0) = y_{0,1} \mbox{ and } y_{1}(T) = 0.
\end{equation}
By~\eqref{eq:mincost}, the optimal solution to~\eqref{eq:defauxiloptim} is given by:
\begin{equation*}
\eta_T(u_T) = y^\dagger_{0,1} W_{J_{1}, C_{1}}(T)^{-1} y_{0,1},
\end{equation*}
where the optimal control $u_T$ is
\begin{equation}\label{eq:etatut}
u_T(t) = - C^\dagger_{1} e^{ -J^\dagger_{1} t} W_{J_{1}, C_{1}}(T)^{-1} y_{0,1} , \quad \mbox{for } 0 \leq t \leq T.
\end{equation}
With some abuse of notation, we extend the domain of $u_T$ to $[0,\infty)$ by setting $u_T(t) := 0$ for all  $t > T$.  It is clear that $u_T\in \rL^2([0,\infty),\C^m)$ for all $T > 0$.

We claim that $u_T \in \cU(x_0)$, i.e., the solution $x(t)$ of system~\eqref{eq:linsys} generated by $u_T(t)$ converges to $0$ as $t\to\infty$.  
Because $x(t)$ and $y(t)$ are related by similarity transformation, it suffices to show that the solution $y(t)$ driven by $u_T(t)$ converges to $0$.  Recall that the dynamics of $y(t)$ obey~\eqref{eq:tildexdynamics}. In particular, the dynamics $y_{1}(t)$ and $y_{\mathsf{a}}(t)$ are decoupled if $u(t) = 0$. 
By construction of $u_T(t)$, the solution $y(t)$ at time $T$ is
$$
y(T) = \begin{bmatrix}
    y_{1}(T)\\
    y_{\mathsf{a}}(T)
\end{bmatrix} = 
\begin{bmatrix}
0 \\
y_{\mathsf{a}}(T)
\end{bmatrix}.
$$
where $y_{\mathsf{a}}(T) = e^{J_{\mathsf{a}}T}y_{0,\mathsf{a}} + \int_0^T e^{J_{\mathsf{a}}(T - t)} C_{\mathsf{a}} u_T(t) \rrd t$. 
Then, for $t \geq T$, we have 
$$
y(t) = \begin{bmatrix}
    0\\
    e^{J_{\mathsf{a}}(t-T)} y_{\mathsf{a}}(T)
\end{bmatrix}.
$$
The claim now follows  from the fact that $J_{\mathsf{a}}$ is Hurwitz by definition, which ensures that $\lim_{t \to \infty} y_{\mathsf{a}}(t) = 0$.

By the hypothesis that $\widetilde u^*$ solves Problem~\ref{prob:infinitehorizon}, we have that 
$\eta(\widetilde u^*)\leq \eta(u_T) =  y^\dagger_{0,1} W_{J_{1}, C_{1}}(T)^{-1} y_{0,1}$ for all $T > 0$. We conclude that
$$
\eta(\widetilde u^*) \leq \lim_{T\to\infty} y^\dagger_{0,1} W_{J_{1}, C_{1}}(T)^{-1} y_{0,1} = y^\dagger_{0,1} 
\begin{bmatrix}
W_{\hspace{-.04cm}\mathsf{u}}^{-1} & 0 \\
0 & 0
\end{bmatrix}
y_{0,1} = y_0^\dagger \widetilde K y_0 = x_0^\dagger K x_0,
$$
where the first equality follows from Theorem~\ref{thm:main} (with the $(A, B)$ pair replaced with the $(J_{1}, C_{1})$ pair), the second equality follows from the definition of $\widetilde K$ in~\eqref{eq:deftildeK}, and the last equality follows from the definition of $K$ in~\eqref{eq:limitK} and the relation $y_0 = P^{-1} x_0$.
\end{proof}

With the two previous lemmas, we now prove the following proposition.

\begin{proposition}
If $x_0\notin  \bV_{\hspace{-.04cm}\mathsf{a}}$,  then there is no solution to Problem~\ref{prob:infinitehorizon}.  
\end{proposition}

\begin{proof}
The proof will be carried out by contradiction. Suppose that  Problem~\ref{prob:infinitehorizon} has a solution, which we denote by $\eta(\widetilde u^*)$;  
clearly, the control law $\widetilde u^*$ drives $y_{\mathsf{u}}(t)$ (see~\eqref{eq:y1indep}) to~$0$ and meets the requirements of  optimal control problem~\eqref{eq:minenergyprobunstable}.

Let $u^*$ be the unique minimizer of problem~\eqref{eq:minenergyprobunstable} as stated in item~1 of Lemma~\ref{lem:tildeJK}.  Because $\widetilde u^* \in \cU(x_0)$, the second item of this lemma yields $\eta(\widetilde u^*) \geq \eta(u^*) = x_0^\dagger K x_0$. 
From  Lemma~\ref{lem:upperboundforJ}, we have the reverse inequality $\eta(\widetilde u^*) \leq  x_0^\dagger K x_0$. 
It thus follows that $\eta(\widetilde u^*) = x_0^\dagger K x_0$. Since the optimal control law that minimizes the cost for the problem~\eqref{eq:minenergyprobunstable} is unique by item~1 of Lemma~\ref{lem:tildeJK}, 
we have that 
$$u^*(t) = \widetilde u^*(t) = - C^\dagger \widetilde K y(t) = - B^\dagger K x(t).$$ 
Thus, using the optimal control law $\widetilde u^*$, we have the closed loop system
$\dot x(t) = (A - BB^\dagger K) x(t)$. However, by the hypothesis that $x_0\notin \bV_\mathsf{a}$, the solution $x(t)$ does not converge to $0$ as $t\to\infty$, which implies that $\widetilde u^*$ is not admissible. 
\end{proof}

This completes the proof of Theorem~\ref{thm:main2p}. \hfill{$\qed$}

\section{The case of imaginary eigenvalues}\label{sec:pfimaginary}

We state and prove in this subsection a core result of our proof for Theorem~\ref{thm:main}.  Namely, we establish the divergence rate of the controllability Gramian $W(T)$ as $T \to \infty$, of a controllable system whose state matrix has only {\em imaginary} eigenvalues.

Consider the system $(J,C)$, where $J$ is a matrix with only imaginary eigenvalues and in the Jordan normal form (i.e., with the convention of the previous sections, $J_{\mathsf{o}}$ is now $J$). 
To proceed, we decompose $J$ into its Jordan blocks, which we denote by $M_i \in \C^{d_i\times d_i}$, for $i=1,\ldots,k$. Without loss of generality, we assume that blocks with the same eigenvalue are contiguous in the ordering. We associate to each block $M_i$ the matrices $C_i$ and $D_i$ obtained as follows:
we  decompose  $C$ and write
\begin{equation}\label{eq:defL}
J = 
\begin{bmatrix}
M_1 & &  \\
& \ddots & \\
& & M_k
\end{bmatrix} \quad \mbox{and}\quad C  = 
\begin{bmatrix}
C_1 \\
\vdots \\
C_k
\end{bmatrix},
\end{equation}
where $C_i \in \C^{d_i \times m}$. 
We then introduce the diagonal matrix 
\begin{equation}\label{eq:defbigD}
D(T) := 
\begin{bmatrix}
D_1(T) & & \\
& \ddots & \\
& & D_k(T)
\end{bmatrix}, \quad \mbox{with } D_i(T):= 
\begin{bmatrix}
T^{d_i - 1} & & & \\
& \ddots & & \\
& & T & \\
& & & 1
\end{bmatrix}\in \C^{d_i\times d_i}.
\end{equation}
We now state the main result of this section:

\begin{theorem}\label{thm:pimlimit}
Let $(J,C)$ be a controllable pair, where $J$ is in  Jordan normal form and with only imaginary eigenvalues. 
Then, there exists a nonsingular matrix $S$ such that  
$$
\lim_{T\to \infty} \frac{1}{T} D(T)^{-1} W_{J,C}(T) D(T)^{-1} = S.$$ 
\end{theorem}

By Theorem~\ref{thm:pimlimit}, it is straightforward that 
$$
\lim_{T\to\infty} W_{J, C}(T)^{-1} = \lim_{T\to\infty} \frac{1}{T}D(T)^{-1}S^{-1}D(T)^{-1} = 0,
$$
which establishes Theorem~\ref{thm:main} for the case where $J$ has only imaginary eigenvalues (in this case, we have that $n_{\mathsf{u}} = 0$, so the matrix $K$ given in~\eqref{eq:limitK} is $0$).

\subsection{On repeated eigenvalues}\label{ssec:repeig}
In this subsection, we prove Theorem~\ref{thm:pimlimit} for the case where all the eigenvalues of $J$ are the same, and equal to $\lambda$.

To prove the result, we first derive a property of the $C$ matrix that follows from the controllability of the system $(J,C)$. Let $g_i \in \C^m$ be the {\em last} row of $C_i$, $1 \leq i \leq k$, and gather all such vectors in the matrix $G$: 
$$
G :=
\begin{bmatrix}
g_1 \\
\vdots \\
g_{k}
\end{bmatrix}\in \C^{k \times m}.$$
We have the following result:
\begin{lemma}\label{lem:pairJB}
If the pair $(J,C)$ is controllable, then
$G$ is of full row rank.
\end{lemma}

\begin{proof}
Since $(J,C)$ is controllable, so are the pairs $(M_i, C_i)$ for $1 \leq i \leq k$. By the Hautus test, the following matrix has full row rank:
$$
\begin{bmatrix}
\lambda I -  J &   C 
\end{bmatrix} = 
\begin{bmatrix}
\lambda I - M_{ 1} & & & C_{1} \\
& \ddots & & \vdots \\
& & \lambda I - M_{k} & C_{k}
\end{bmatrix}.
$$
Now, observe that $(\lambda I -  M_{i})$ are matrices with all zero entries save for the upper-diagonal, which has entries one. Hence, the last rows of  $(\lambda I -  M_{i})$ are all zeros. Thus, for the matrix $\begin{bmatrix}
\lambda I -  J &   C 
\end{bmatrix}$ to have full row rank, it is necessary that the last rows of the $C_{i}$'s be linearly independent, which proves the result.
\end{proof}

In the  following Lemma, we establish the divergence rate of $W_{J,C}(T)$. 

\begin{lemma}\label{lem:limitSDPD}
It holds that
\begin{equation}\label{eq:DPD}
\lim_{T\to\infty}\frac{1}{T} D(T)^{-1} W_{J,C}(T) D(T)^{-1} = \Delta \Phi \Delta,
\end{equation}
where $\Delta$ and $\Phi$ are given by
$$
\Delta := 
\begin{bmatrix}
\Delta_1 & & \\
& \ddots & \\
& & \Delta_k
\end{bmatrix} \quad \mbox{with }
\Delta_i:=
\begin{bmatrix}\frac{(-1)^{d_i - 1}}{(d_i-1)!} &  &  &  \\
& \frac{(-1)^{d_i - 2}}{(d_i-2)!} & & \\
& & \ddots & \\
& & & 1
\end{bmatrix},
$$
and 
$$
\Phi:= 
\begin{bmatrix}
g_1 g_1^\dagger \Psi_{11} & \cdots & g_1 g_k^\dagger\Psi_{1k} \\
\vdots & \ddots & \vdots \\
g_k g_1^\dagger \Psi_{k1} & \cdots & g_k g_k^\dagger \Psi_{kk}
\end{bmatrix} \quad \mbox{with }
\Psi_{ij} := \left [\frac{ 1 }{(d_i + d_j - \alpha - \beta + 1)} \right ]_{\alpha\beta} 
$$
for $1 \leq \alpha \leq d_i$ and $1 \leq \beta \leq d_j$.
\end{lemma}

\begin{proof}
Let $\widetilde J := J - \lambda I$ and $\widetilde M_i:= M_i - \lambda I$. 
Then, the controllability Gramian $W_{J,C}(T)$ takes the form
\begin{equation*}
W_{J,C}(T) = \int_0^T e^{-\lambda t} e^{-\widetilde J t} C C^\dagger e^{-\widetilde J^{\, \top} t} e^{\lambda t} \rrd t\\=\int_0^T  \underbrace{e^{-\widetilde J t} C  C^\dagger e^{-\widetilde J^{\, \top}t}}_{R(t)}  \rrd t.
\end{equation*}
We decompose the integrand $R(t)$ into blocks according to 
$$
R(t) = [R_{ij}(t)]_{1\leq i,j\leq k}, \quad \mbox{with } R_{ij}(t) := e^{-\widetilde M_i t}  C_i  C_j^\dagger e^{-\widetilde M_j^\top t} \in \C^{d_i\times d_j}.
$$
Recalling that $\widetilde M_i$ is of dimension $d_i\times d_i$ and that $g_i$ is the last row of $C_i$,  we express the product $e^{-\widetilde M_i t} C_i $ isolating the terms of {\em highest} order in $t$ 
\begin{multline}\label{eq:xpeepee}
e^{-\widetilde M_i t} C_i = \begin{bmatrix}
    1 & -t & \frac{t^2}{2}  & \cdots & \frac{ (-t)^{d_i-1}}{(d_i - 1)!}\\
    0 & 1 & -t & \cdots & \frac{(-t)^{d_i-2}}{(d_i - 2)!} \\
    \vdots & & \ddots & \ddots & \vdots \\
    0 & 0 &  & 1 & -t \\
    0 & 0 & \cdots  & 0 & 1 
\end{bmatrix} C_i =  
\begin{bmatrix}
\frac{(-t)^{d_i-1}}{(d_i - 1)!} g_i + O(t^{d_i - 2})   \\
\frac{(-t)^{d_i-2}}{(d_i - 2)!} g_i + O(t^{d_i - 3})  \\
\vdots \\
g_i
\end{bmatrix} \\=  D_i(t) \Delta_i \bfo g_i+ H_i(t),
\end{multline}
where $\bfo$ is the vector of ones, $D_i(t)$ is given in~\eqref{eq:defbigD}, 
and $H_i(t)\in \C^{d_i \times m}$ is such that 
\begin{equation}\label{eq:defHi}
H_i(t)=\begin{bmatrix}
    O(t^{d_i - 2}) \\
    \vdots \\
    O(1) \\
    0
\end{bmatrix}, 
\end{equation}
i.e., $H_i(t)$ is a matrix whose entries in the $\alpha$th row, for $1\leq \alpha \leq d_i - 1$, are of order $O(t^{d_i - \alpha - 1})$ and whose last row is~$0$. 
It follows that
\begin{align}
R_{ij}(t)  =&\, \,  H_i(t)H_j(t)^\dagger.\label{eq:rijhihj}\\ 
&  + \Delta_i D_i(t) \bfo g_i H_j(t)^\dagger \label{eq:rijcross1} \\ 
& + H_i(t) g_j^\dagger \bfo^\top  D_j(t) \Delta_j \label{eq:rijcross2} \\
& + \Delta_i D_i(t) \bfo g_i g_j^\dagger \bfo^\top  D_j(t) \Delta_j. \label{eq:rijzeroorder}
\end{align}
For our purpose, it is sufficient to only describe only terms of the highest order in $t$. 
For~\eqref{eq:rijhihj}, we use~\eqref{eq:defHi} to obtain that
$$
\int_0^T  H_i(t) H_j(t)^\dagger  \rrd t = 
\begin{bmatrix}
O(T^{d_i+d_j- 3}) & \cdots & O(T^{d_i-1}) & 0 \\
\vdots & \ddots & \vdots  & 0 \\
O(T^{d_j - 1}) & \cdots & O(T) & 0 \\
0 & \cdots & 0 & 0 
\end{bmatrix}=:I_1,
$$

For~\eqref{eq:rijcross1}, we first use~\eqref{eq:defHi} to obtain 
$$
 g_i H_j(t)^\dagger = 
\begin{bmatrix} O(t^{d_j-2}) &O(t^{d_j-1}) & \cdots & O(1) & 0 
\end{bmatrix}.
$$
Then, using the definition of $D_i(t)$ (see~\eqref{eq:defbigD}), we have that
$$
\Delta_i D_i(t) \bfo g_i H_j(t)^\dagger = 
\begin{bmatrix}
O(T^{d_i+d_j- 3}) & \cdots & O(T^{d_i-1}) & 0 \\
\vdots & \ddots & \vdots  & \vdots \\
O(T^{d_j - 2})& \cdots & O(1) & 0 
\end{bmatrix},
$$
and hence, 
$$
\int_0^T \Delta_i D_i(t) \bfo g_i H_j(t)^\dagger  \rrd t = 
\begin{bmatrix}
O(T^{d_i+d_j- 2}) & \cdots & O(T^{d_i}) & 0 \\
\vdots & \ddots & \vdots  & \vdots \\
O(T^{d_j - 1})& \cdots & O(T) & 0 
\end{bmatrix}=:I_2.
$$

For~\eqref{eq:rijcross2}, we similarly have that
$$
\int_0^T  H_i(t) g_j^\dagger \bfo^\top  D_j(t) \Delta_j  \rrd t = 
\begin{bmatrix}
O(T^{d_i+d_j- 2}) & \cdots & O(T^{d_i - 1}) \\
\vdots & \ddots &  \vdots \\
O(T^{d_j}) & \cdots  & O(T) \\
0 & \cdots & 0 
\end{bmatrix}=:I_3.
$$

With the above definitions of  $I_1$, $I_2$, and $I_3$, a straightforward computation yields
$$\frac{1}{T} D_i(T)^{-1}\left [I_1+I_2+I_3\right]   D_j(T)^{-1} = O(T^{-2}) + O(T^{-1}) + O(T^{-1}).$$  

Finally, for the term~\eqref{eq:rijzeroorder}, note that $g_ig_j^\dagger$ is a scalar and that
\begin{multline*}
\int_0^T D_i(t) \bfo \bfo^\top D_j(t) \rrd t  = \left [ \int_0^T   T^{d_i + d_j - \alpha - \beta } \rrd t  \right ]_{\alpha\beta}  = \left [ \frac{T^{d_i + d_j - \alpha - \beta + 1}}{d_i + d_j - \alpha - \beta + 1} \right ]_{\alpha\beta}\\ = 
T D_i(T) \Psi_{ij} D_j(T),
\end{multline*}
where $1\leq \alpha \leq d_i$ and $1\leq \beta \leq d_j$. 
Thus, 
\begin{equation}\label{eq:limrijt}
\frac{1}{T}D_i(T)^{-1}\left[\int_0^T R_{ij}(t) \rrd t \right]D_j(T)^{-1}= \\
g_ig_j^\dagger \Delta_i \Psi_{ij}\Delta_j + O(T^{-1}). 
\end{equation}
The lemma then follows directly from~\eqref{eq:limrijt}. 
\end{proof}

Note that the diagonal matrix $\Delta$ is invertible. Thus, to establish Theorem~\ref{thm:pimlimit}, it remains to show that the matrix $\Phi$ is invertible as well. Our proof relies on introducing an auxiliary linear time-invariant system, which is controllable and whose controllability Gramian at $T = \infty$ is equal to $\Phi$. 

To this end, we define the system pair $(\Pi, \Gamma)$, which has the same dimension as $(J,C)$, as follows: 
$$
\Pi := 
\begin{bmatrix}
\Pi_1 & & \\
& \ddots & \\
& & \Pi_k
\end{bmatrix}, \quad \mbox{with }
\Pi_i := 
\frac{1}{2}
\begin{bmatrix}
2d_i - 1 & & & \\
& 2d_i-3 & & \\
& & \ddots & \\
& & & 1
\end{bmatrix}
\quad \mbox{and} \quad 
\Gamma := \begin{bmatrix} \bfo g_1 \\ \vdots \\ \bfo g_k
\end{bmatrix},$$ 
where $\Pi_i \in \R^{d_i \times d_i}$ and $\bfo g_i \in \C^{d_i \times m}$, for $1 \leq i \leq k$.

\begin{lemma}\label{lem:infsys}
The pair $(\Pi,\Gamma)$ is controllable and, moreover, 
\begin{equation}\label{eq:Wpigamma}
W_{\Pi,\Gamma}(\infty)=\Phi,
\end{equation} 
where $\Phi$ is given in Lemma~\ref{lem:limitSDPD}. 
\end{lemma}

\begin{proof}
We first show that $(\Pi, \Gamma)$ is controllable and then establish~\eqref{eq:Wpigamma}.

\xc{Proof that $(\Pi, \Gamma)$ is controllable.}
    We use the Hautus test to establish controllability. Precisely, we show that for any eigenvalue $\mu$ of $\Pi$, 
    \begin{equation}\label{eq:hautustest}
    \rank \begin{bmatrix} \mu I - \Pi & \Gamma \end{bmatrix} = n.
    \end{equation}
    We have that $\Pi$ is diagonal with each block $\Pi_i$ having pairwise distinct eigenvalues.  
    Also, it should be clear that the eigenvalues of $\Pi$ are repeated {\em at most} $k$ times. 
    Thus, the diagonal matrix $(\mu I - \Pi)$ 
    has at most $k$ zero rows, say $k'\leq k$, with at most one zero row per block $(\mu I - \Pi_i)$.
    These rows correspond in the matrix given in~\eqref{eq:hautustest}  
    to rows of the form 
    \begin{equation}\label{eq:zerogi}
    \begin{bmatrix} 0 & g_{i_j}\end{bmatrix} \quad \mbox{for some pairwise distinct } i_1, \ldots, i_{k'} \in \{1,\ldots,k\}.
    \end{equation} 
    The other rows in this matrix are clearly linearly independent---since the corresponding rows in $(\mu I - \Pi)$ are linearly independent---and independent of the rows in~\eqref{eq:zerogi}. 

    It remains to show that the rows in~\eqref{eq:zerogi} are linearly independent. 
    But this holds because the $i_j$'s  are pairwise distinct and, by Lemma~\ref{lem:pairJB}, the $g_i$'s are linearly independent. Together, these two facts yield that the row vectors in~\eqref{eq:zerogi} are also linearly independent. This proves that the matrix in~\eqref{eq:hautustest} has full row rank and hence, 
    the pair $(\Pi,\Gamma)$ is controllable.

 \xc{Proof that~\eqref{eq:Wpigamma} holds.} First, note that the controllability Gramian
 $$W_{\Pi,\Gamma}(\infty) = \int_0^\infty e^{-\Pi t} \Gamma \Gamma^\dagger e^{-\Pi t} \rrd t$$ 
 exists since $-\Pi$ is Hurwitz. 
 We now evaluate it block by block (the blocks match the decomposition of $\Phi$ given above in the statement of Lemma~\ref{lem:limitSDPD}). To this end, we evaluate the product
 \begin{align*}
    e^{-\Pi_i t}\bfo g_i g_j^\dagger \bfo^\top e^{-\Pi_j t} & =  g_ig_j^\dagger \left[ e^{-(d_i+d_j-\alpha-\beta+1) t} \right]_{\alpha\beta},
\end{align*}
for $1\leq \alpha \leq d_i$ and $1 \leq \beta \leq d_j$. Note that $(d_i + d_j - 
\alpha - \beta + 1) \geq 1$ for all $\alpha$ and $\beta$. 
Integrating the above yields 
\begin{align*}
    \int_0^\infty e^{-\Pi_i t}\bfo g_i g_j^\dagger \bfo^\top e^{-\Pi_j t} & =  g_ig_j^\dagger \Psi_{ij}.
\end{align*}
This completes the proof. 
\end{proof}

Theorem~\ref{thm:pimlimit} for  the case of Jordan blocks with the same eigenvalues is then an immediate consequence of Lemmas~\ref{lem:limitSDPD} and~\ref{lem:infsys}, and the fact that the controllability Gramian $W_{\Pi,\Gamma}(\infty)$ is of full rank. \hfill{$\qed$}

\subsection{Proof of Theorem~\ref{thm:pimlimit}}

We now consider the general case where the matrix $J$  has Jordan blocks  with imaginary, but not necessarily the same, eigenvalues. 

Let $r\leq k$ be the number of {\em distinct} eigenvalues of $J$, and let $k_1,\ldots, k_r$ be the numbers of blocks $M_j$ with eigenvalues $\lambda_1,\ldots,\lambda_r$, respectively.
Let $s_0:= 0$ and $s_p:= \sum_{q = 1}^{p} k_q$, for $p = 1,\ldots, r$. We let
$$
\bM_p := 
\begin{bmatrix}
M_{s_{p-1} + 1} & & \\
& \ddots & \\
& & M_{s_p}
\end{bmatrix}, \,  
\bC_p := 
\begin{bmatrix}
C_{s_{p-1} + 1} \\
\vdots \\
C_{s_p}
\end{bmatrix}, \mbox{ and } 
\bD_p:= 
\begin{bmatrix}
D_{s_{p-1} + 1} & & \\
& \ddots & \\
& & D_{s_p}
\end{bmatrix}.
$$ 
In words, $\bM_p$ contains all the Jordan blocks $M_j$'s with eigenvalue $\lambda_p$, and $\bC_p$ and $\bD_p$ contain the corresponding $C_j$'s and $D_j$'s, for $1 \leq p \leq r$, respectively, where $D_j$ is as in~\eqref{eq:defbigD} where we omit the dependence on $T$.

We next decompose the matrix $W_{J,C}(T)$ according to the grouping of Jordan blocks by  eigenvalues introduced above, namely,  we set 
$$
W_{J,C}(T) =
\begin{bmatrix}
L_{11} & \cdots & L_{1r}\\
\vdots & \ddots & \vdots \\
L_{r1} & \cdots & L_{rr} 
\end{bmatrix},
$$
where
$$
L_{pq} := \int_0^T e^{-\bM_p t} \bC_p \bC_q^\dagger e^{-\bM_q^\dagger t} \rrd t.  
$$
We show that the $pq$th blocks of $S$ exist and compute them explicitly. We consider two cases, first dealing with the diagonal blocks (in Lemma~\ref{lem:diagS}) and then the off-diagonal blocks (in Lemma~\ref{lem:offdiagS}).

\begin{lemma}\label{lem:diagS}
For every $p = 1,\ldots, r$, 
$$
S_{pp} = \lim_{T\to\infty} \frac{1}{T} \bD_p(T)^{-1} L_{pp}(T) \bD_p(T)^{-1}
$$
is nonsingular. 
\end{lemma}

\begin{proof}
Note that $L_{pp}(T)$ is the controllability Gramian associated with $(\bM_p, \bC_p)$, and that by construction, all the eigenvalues of $\bM_p$ are the same. 
The lemma thus follows directly from the result of Subsection~\ref{ssec:repeig}. 
\end{proof}

We now evaluate $S_{pq}$, for $p\neq q$, and have the following lemma: 

\begin{lemma}\label{lem:offdiagS}
For $p,q \in \{ 1,\ldots, r\}$ with $p \neq q$, we have 
$$
S_{pq} = \lim_{T\to\infty} \frac{1}{T} \bD_{p}(T)^{-1} L_{pq}(T) \bD_{q}(T)^{-1} = 0.
$$
\end{lemma}

\begin{proof}
Let $\widetilde \bM_p:= \bM_p - \lambda_p I$. 
We have that
\begin{equation*}
L_{pq}(T)  = \int_0^T e^{-\bM_p t} \bC_p \bC_q^\dagger e^{-\bM_q^\dagger t} \rrd t  = \int_0^T \underbrace{e^{(\lambda_q-\lambda_p) t} e^{-\widetilde \bM_p t} \bC_p \bC_q^\dagger e^{-\widetilde \bM_q^\top t}}_{R_{pq}(t)} \rrd t.
\end{equation*}
Decompose $R_{pq}(t)$ as
$$
R_{pq}(t) =
\begin{bmatrix}
Q_{11}(t) & \cdots & Q_{1k_q}(t) \\
\vdots & \ddots & \vdots \\
Q_{k_p 1}(t) & \cdots & Q_{k_pk_q}(t) 
\end{bmatrix}
\quad \mbox{with }
Q_{ij}(t):= e^{(\lambda_q-\lambda_p) t} e^{-\widetilde M_{i'} t} C_{i'} C_{j'}^\dagger e^{-\widetilde M_{j'}^\top t}, 
$$
where the subindices $i'$ and $j'$ in the expression of $Q_{ij}(t)$ are given by
$$i':=s_{p-1}+i \quad \mbox{and} \quad j':=s_{q-1}+j,$$ 
for $1 \leq i \leq k_p$ and $1 \leq j \leq k_q$.

We have evaluated the product $e^{-\widetilde M_{i'}t}C_{i'}$ in~\eqref{eq:xpeepee}. It follows that the $\alpha\beta$th entry of $Q_{ij}(t)$, for $1 \leq \alpha \leq d_{i'}$ and $1 \leq \beta \leq d_{j'}$,  is given by
$$
Q_{ij,\alpha\beta}(t) =
e^{(\lambda_q - \lambda_p) t} 
\sum_{\ell = 0}^{d_{i'} + d_{j'} - \alpha - \beta}\gamma_{\ell} t^{\ell},
$$
for some coefficients $\gamma_\ell$. Note that $\lambda_q\neq \lambda_p$ since $p \neq q$. Since both $\lambda_p$ and $\lambda_q$ are imaginary, we obtain that for any $\ell \geq 0$, 
$$
\int_0^T e^{(\lambda_q - \lambda_p) t} t^\ell \rrd t = 
\begin{cases}
\frac{e^{(\lambda_q - \lambda_p)T}}{\lambda_q - \lambda_p} T^\ell + O(T^{\ell-1}) & \mbox{if } \ell  \geq 1 \vspace{.1cm}\\
\frac{e^{(\lambda_q - \lambda_p)T} -1 }{\lambda_q - \lambda_p} & \mbox{if } \ell = 0,
\end{cases}
$$
which implies that
$$
L_{pq,ij} = \left[\int_0^T Q_{ij,\alpha\beta}(t) \rrd t = O(T^{d_{i'} + d_{j'} - \alpha - \beta})\right]_{\alpha\beta},
$$
for $1\leq \alpha \leq d_{i'}$ and $1\leq \beta \leq d_{j'}$. 
We thus conclude that
$$
\frac{1}{T} D_{i'}(T)^{-1} L_{pq,ij}(T)D_{j'}(T)^{-1} = O(T^{-1}),
$$
and hence, 
$$
S_{pq}=\lim_{T\to\infty} \frac{1}{T} \bD_{p}(T)^{-1} L_{pq}(T)\bD_{q}(T)^{-1} = 0.
$$
This establishes the result. 
\end{proof}

Theorem~\ref{thm:pimlimit} is then an immediate consequence of Lemmas~\ref{lem:diagS} and~\ref{lem:offdiagS}.
\hfill{\qed}

\section{The case of non-positive eigenvalues}\label{sec:pfnonpositive}
In this section, we show that Theorem~\ref{thm:main} holds in the case where the matrix $A$ has only eigenvalues with {\em non-positive} real parts. 
Following our convention, since $A$ has no eigenvalues with positive real parts, the matrix $J_{\mathsf{u}}$ is empty and $J = J_{\mathsf{s}} = \diag(J_{\mathsf{o}}, J_{\mathsf{a}})$, where the real parts of the eigenvalues of $J_{\mathsf{o}} \in \C^{n_{\mathsf{o}} \times n_{\mathsf{o}}}$ and $J_{\mathsf{a}}\in \C^{n_{\mathsf{a}} \times n_{\mathsf{a}}}$ are zero and negative, respectively. Similarly, $C_{\mathsf{u}}$ is empty and we decompose $C = C_{\mathsf{s}} =[C_{\mathsf{o}}; C_{\mathsf{a}}]$.

To proceed, we  partition the controllability Gramian $W_{J,C}(T)$ into $4$ blocks which agree with the decomposition of $J_{\mathsf{s}}$ in $J_{\mathsf{o}}, J_{\mathsf{a}}$:
\begin{equation}\label{eq:defV2}
W_{J,C}(T) =:V_{\mathsf{s}}(T)= 
\begin{bmatrix}
V_{\mathsf{o}}(T) & V_{\mathsf{oa}}(T) \\
V_{\mathsf{oa}}(T)^\dagger & V_{\mathsf{a}}(T) 
\end{bmatrix},
\end{equation}
where
\begin{equation}\label{eq:defv2i}
    \begin{aligned}
        V_{*}(T)& := W_{J_{*}, C_{*}}(T) = \int_0^T e^{-J_{*} t} C_{*} C^\dagger_{*} e^{-J^\dagger_{*} t} \rrd t, \quad \mbox{for } * = \mathsf{o}, \mathsf{a} \\
        V_{\mathsf{oa}}(T) & :=  \int_0^T e^{-J_{\mathsf{o}} t} C_{\mathsf{o}} C^\dagger_{\mathsf{a}} e^{-J^\dagger_{\mathsf{a}} t} \rrd t.
    \end{aligned}
\end{equation}
Using the Schur complement~\cite{horn2012matrix}, we obtain from~\eqref{eq:defV2} that
\begin{equation}\label{eq:schurcompV2}
V_{\mathsf{s}}^{-1} = 
\begin{bmatrix}
 V_{\mathsf{o}}^{-1}+V_{\mathsf{o}}^{-1} V_{\mathsf{oa}} (V_{\mathsf{s}}/V_{\mathsf{o}})^{-1} V_{\mathsf{oa}}^\dagger V_{\mathsf{o}}^{-1} & -V_{\mathsf{o}}^{-1} V_{\mathsf{oa}} (V_{\mathsf{s}}/V_{\mathsf{o}})^{-1}   \\
- (V_{\mathsf{s}}/V_{\mathsf{o}})^{-1}  V_{\mathsf{oa}}^\dagger V_{\mathsf{o}}^{-1}  & (V_{\mathsf{s}}/V_{\mathsf{o}})^{-1}
\end{bmatrix},
\end{equation}
where we have omitted the argument $T$ and $V_{\mathsf{s}}/V_{\mathsf{o}}$ is given by
\begin{equation}\label{eq:schurcompV221}
V_{\mathsf{s}}/V_{\mathsf{o}} := V_{\mathsf{a}} - V_{\mathsf{oa}}^\dagger V_{\mathsf{o}}^{-1} V_{\mathsf{oa}}.
\end{equation} 
The main result of this section is the following:

\begin{theorem}
\label{th:theorem6prop2}
Let $(J,C)$ be a controllable pair, where $J$ is in Jordan normal form and has only eigenvalues with non-positive real parts. Let $V_{\mathsf{s}}(T)$ and the blocks $V_{*}$, for $* = \mathsf{o},\mathsf{a},\mathsf{oa}$, be given as in~\eqref{eq:defV2} and~\eqref{eq:defv2i}, respectively.  
Then, the following items hold:
\begin{enumerate}
\item $\lim_{T\to\infty}(V_{\mathsf{s}}/V_{\mathsf{o}})^{-1} = 0$.
\item $\lim_{T\to\infty} (V_{\mathsf{s}}/V_{\mathsf{o}})^{-1} V_{\mathsf{oa}}^\dagger V_{\mathsf{o}}^{-1} = 0$.
\item $\lim_{T\to\infty} \left (V_{\mathsf{o}}^{-1} + V_{\mathsf{o}}^{-1} V_{\mathsf{oa}}(V_{\mathsf{s}}/V_{\mathsf{o}})^{-1} V_{\mathsf{oa}}^\dagger V_{\mathsf{o}}^{-1}\right ) = 0$.
\end{enumerate}  
\end{theorem}

The following result is now an immediate consequence of Theorem~\ref{th:theorem6prop2}:

\begin{corollary}\label{cor:theorem6prop2}
    If $(A, B)$ is controllable and if $A$ only has  eigenvalues with non-positive real parts, then 
    $$
    \lim_{T\to\infty} W_{(A,B)}(T)^{-1} = 0.
    $$
\end{corollary}

\begin{proof}
Using the relation $J = P^{-1} A P$ and $C = P^{-1} B$, Theorem~\ref{th:theorem6prop2}, together with~\eqref{eq:schurcompV2}, we have that $$\lim_{T\to\infty} W_{A, B}(T)^{-1} =\lim_{T\to\infty} (P^{-1})^\dagger W_{J, C}(T)^{-1} P^{-1} = 0.$$ 
\end{proof}

\subsection{Asymptotic convergence of matrix products}\label{ssec:buff}

A major technical challenge in proving Theorem~\ref{th:theorem6prop2} is to deal with the asymptotic behavior of the terms $V_{\mathsf{oa}}^\dagger V_{\mathsf{o}}^{-1}$ and $V_{\mathsf{oa}}^\dagger V_{\mathsf{o}}^{-1} V_{\mathsf{oa}}$. Note that these terms appear {\em only if} the system has both imaginary eigenvalues and eigenvalues with negative real parts.
To elaborate, while one can use Theorem~\ref{thm:pimlimit} to argue that $V_{\mathsf{o}}^{-1}$ vanishes as $T\to\infty$, the term $V_{\mathsf{oa}}$ diverges as $-J_{\mathsf{a}}$ has eigenvalues with positive real parts. Thus, to establish convergence (or divergence) of their product, we cannot rely on the convergence of the individual  terms. 

Even more, as will be clear to see below, only very specific groupings of the terms are so that each product of them converges asymptotically, and defining these groupings requires us to add ``buffer terms'' (e.g., terms such as $e^{J_{\mathsf{o}}T}$, $e^{J^\dagger_{\mathsf{o}}T}$, and $e^{J^\dagger_{\mathsf{a}}T}$ in~\eqref{eq:v21v23} below) between the original terms of product.

This approach to establishing convergence or divergence of matrix products will be used extensively throughout the remainder of the paper.

\begin{proposition}\label{prop:inthreeterms}
The following two items hold: 
\begin{enumerate}
\item We have that
\begin{equation*}
\lim_{T\to\infty} e^{J_{\mathsf{a}}T} (V_{\mathsf{s}}/V_{\mathsf{o}}) e^{J_{\mathsf{a}}^\dagger T} = \int_0^\infty e^{J_{\mathsf{a}}t} C_{\mathsf{a}} C_{\mathsf{a}}^\dagger e^{J_{\mathsf{a}}^\dagger t} \rrd t=W_{-J_{\mathsf{a}},{C_{\mathsf{a}}}}(\infty).
\end{equation*}
\item We define 
\begin{equation}\label{eq:v21v23}
\widetilde V_{\mathsf{o}}  := e^{J_{\mathsf{o}} T}V_{\mathsf{o}}e^{J_{\mathsf{o}}^\dagger T} \quad \mbox{and} \quad \widetilde V_{\mathsf{oa}}  := e^{J_{\mathsf{o}} T}V_{\mathsf{oa}}e^{J_{\mathsf{a}}^\dagger T}.
\end{equation} 
Then, 
\begin{equation}\label{eq:limeJ22V23}
\lim_{T\to\infty} \widetilde V_{\mathsf{oa}}^\dagger  \widetilde V_{\mathsf{o}}^{-1} e^{J_{\mathsf{o}}T} = 0.
\end{equation}
\end{enumerate}
\end{proposition}

\begin{proof}We prove the two items separately.

\xc{Proof of item~1.} We have that $$e^{J_{\mathsf{a}}T}(V_{\mathsf{s}}/V_{\mathsf{o}})e^{J^\dagger_{\mathsf{a}} T} =e^{J_{\mathsf{a}}T}V_{\mathsf{a}}e^{J^\dagger_{\mathsf{a}} T} - e^{J_{\mathsf{a}}T}V_{\mathsf{oa}}^\dagger V_{\mathsf{o}}^{-1} V_{\mathsf{oa}}e^{J^\dagger_{\mathsf{a}} T}.$$ 
For the first term in the above equation, we use the fact that $V_{\mathsf{a}}(T) = W_{J_{\mathsf{a}},C_{\mathsf{a}}}(T)$ and obtain that
\begin{multline*}
\lim_{T\to\infty} e^{J_{\mathsf{a}} T} V_{\mathsf{a}} e^{J_{\mathsf{a}}^\dagger T} = \lim_{T\to\infty}\int_0^T e^{J_{\mathsf{a}}(T - t)} C_{\mathsf{a}} C_{\mathsf{a}}^\dagger e^{J_{\mathsf{a}}^\dagger (T - t)} \rrd t\\
=\lim_{T\to\infty}  \int_0^T e^{J_{\mathsf{a}}t'} C_{\mathsf{a}} C_{\mathsf{a}}^\dagger e^{J_{\mathsf{a}}^\dagger t'} \rrd t' = W_{-J_{\mathsf{a}}, C_{\mathsf{a}}}(\infty),
\end{multline*}
where we set $t':=T-t$. For the second term,  we write that  
$$
 e^{J_{\mathsf{a}}T} V_{\mathsf{oa}}^\dagger V_{\mathsf{o}}^{-1} V_{\mathsf{oa}} e^{J_{\mathsf{a}}^\dagger T}  = 
 \left [e^{J_{\mathsf{a}}T} V_{\mathsf{oa}}^\dagger e^{J_{\mathsf{o}}^\dagger T}\right ] \left [e^{J_{\mathsf{o}} T} V_{\mathsf{o}} e^{J_{\mathsf{o}}^\dagger T}\right]^{-1} \left [ e^{J_{\mathsf{o}}T} V_{\mathsf{oa}} e^{J_{\mathsf{a}}^\dagger T} \right ]   = 
 \widetilde V_{\mathsf{oa}}^\dagger \widetilde V_{\mathsf{o}}^{-1} \widetilde V_{\mathsf{oa}}. 
$$
If we show that $\widetilde V_{\mathsf{oa}}^\dagger \widetilde V_{\mathsf{o}}^{-1} \widetilde V_{\mathsf{oa}}$ vanishes as $T\to\infty$, then item~1 is established. 
We do so by showing that 
\begin{enumerate}
\item[$(i)$] $\widetilde V_{\mathsf{o}}^{-1}$ vanishes as $T\to\infty$; and
\item[$(ii)$] $\widetilde V_{\mathsf{oa}}$ is bounded as $T\to\infty$. 
\end{enumerate}
\begin{description}
\item[\it Proof of $(i)$.] For the term  $\widetilde V_{\mathsf{o}}$, we can use the same change of variable ($t'= T -t$) as above to express it as follows:  
\begin{equation}\label{eq:vtilde21}
\widetilde V_{\mathsf{o}}(T) = \int_0^T e^{J_{\mathsf{o}} (T - t)} C_{\mathsf{o}} C_{\mathsf{o}}^\dagger e^{J_{\mathsf{o}}^\dagger (T -t)}\rrd t = W_{-J_{\mathsf{o}}, C_{\mathsf{o}}}(T).  
\end{equation}
We need to appeal to  Theorem~\ref{thm:pimlimit} to evaluate the limit of its inverse. To this end,  we  decompose $J_{\mathsf{o}}$ into its Jordan blocks similarly to what was done in~\eqref{eq:defL}:  
\begin{equation}\label{eq:decompJ21}
J_{\mathsf{o}} = \diag(M_1,\ldots,M_{k}),
\end{equation} 
where each $M_i$ is a Jordan block of dimensions $d_i$, for $i=1,\ldots, k$. 
Let $D(T) \in \R^{n_{\mathsf{o}}\times n_{\mathsf{o}}}$ be defined in the same way as~\eqref{eq:defbigD}, which we reproduce here: 
\begin{equation}\label{eq:defDTJ21}
D(T) = \diag(D_1(T),\ldots, D_{k}(T)), \quad \mbox{with } 
D_i =\diag(T^{d_i-1},\ldots, 1). 
\end{equation}  
Because $(-J_{\mathsf{o}}, C_{\mathsf{o}})$ is controllable, with the eigenvalues of $-J_{\mathsf{o}}$ being imaginary, Theorem~\ref{thm:pimlimit} states that 
\begin{equation}\label{eq:v21dt-1}
\lim_{T\to \infty} \frac{1}{T} D(T)^{-1} \widetilde V_{\mathsf{o}}(T) D(T)^{-1} = S,
\end{equation}
for some nonsingular matrix $S$.
This implies that
\begin{equation}\label{eq:v21-1}
\lim_{T\to\infty} \widetilde V_{\mathsf{o}}(T)^{-1} = \lim_{T\to\infty} \frac{1}{T} D(T)^{-1}S^{-1}D(T)^{-1} = 0.
\end{equation}

\item[\it Proof of $(ii)$.] For the term $\widetilde V_{\mathsf{oa}}$, we have
\begin{equation}\label{eq:vtilde23}
\widetilde V_{\mathsf{oa}}  = \int_{0}^T e^{J_{\mathsf{o}} (T - t)} C_{\mathsf{o}} C_{\mathsf{a}}^\dagger e^{J_{\mathsf{a}}^\dagger (T -t)} \rrd t = \int_0^T e^{J_{\mathsf{o}} t} C_{\mathsf{o}} C_{\mathsf{a}}^\dagger e^{J_{\mathsf{a}}^\dagger t} \rrd t.  
\end{equation}
It is not too hard to see that  $\widetilde V_{\mathsf{oa}}$  remains bounded as $T\to\infty$; indeed, the entries of $e^{J_{\mathsf{a}}^\dagger t}$ decay exponentially fast while the entries of $e^{J_{\mathsf{o}}t}$ grow at most polynomially. 
\end{description}
We conclude that 
$$\lim_{T\to\infty} \widetilde V_{\mathsf{oa}}^\dagger \widetilde V_{\mathsf{o}}^{-1} \widetilde V_{\mathsf{oa}} = 0.$$ 
This completes the proof of the first item.

\xc{Proof of item~2.}
As argued in the proof of $(ii)$ in the first item, $\widetilde V_{\mathsf{oa}}$ is bounded in the limit. 
It thus suffices to show that 
\begin{equation}\label{eq:v21j2t}
\lim_{T\to\infty}\widetilde V_{\mathsf{o}}^{-1} e^{J_{\mathsf{o}}T} = 0.
\end{equation}
To this end, recall the decomposition of $J_{\mathsf{o}}$ and the expression of $D(T)$ given in~\eqref{eq:decompJ21} and~\eqref{eq:defDTJ21}, respectively. 
Now, we write 
\begin{equation}\label{eq:anotherthreeterms}
\widetilde V_{\mathsf{o}}^{-1} e^{J_{\mathsf{o}} T} = \left [ TD(T)\right ]^{-1} \left [\frac{1}{T} D(T)^{-1} \widetilde V_{\mathsf{o}}(T) D(T)^{-1} \right ]^{-1} \left [ D(T)^{-1} e^{J_{\mathsf{o}} T}\right ], 
\end{equation}
The first term on the right hand side of~\eqref{eq:anotherthreeterms} converges to $0$ as $T\to\infty$. The second term is bounded by Theorem~\ref{thm:pimlimit} (see also~\eqref{eq:v21dt-1}). 
It is thus enough to show that the last term $D(T)^{-1} e^{J_{\mathsf{o}}T}$ is bounded in the limit to establish~\eqref{eq:v21j2t}.   
But this follows straightforwardly by computation; indeed,  
$$
D_i(T)^{-1} e^{M_i T} = e^{\lambda_i T}
\begin{bmatrix}
    \nicefrac{1}{T^{d_i - 1}} & \nicefrac{1}{T^{d_i - 2}} & \nicefrac{1}{2T^{d_i - 3}}  & \cdots & \nicefrac{1}{(d_i - 1)!}\\
    0 & \nicefrac{1}{T^{d_i - 2}} & \nicefrac{1}{T^{d_i - 3}} & \cdots & \nicefrac{1}{(d_i - 2)!} \\
    \vdots & & \ddots & \ddots & \vdots \\
    0 & 0 &  & \nicefrac{1}{T} & 1 \\
    0 & 0 & \cdots  & 0 & 1 
\end{bmatrix}.
$$ 
Since $\lambda_i$ is  imaginary, we have that every entry of $D_i(T)^{-1} e^{M_i T}$ is absolutely bounded in the limit as $T \to \infty$. 
\end{proof}

\subsection{Proof of Theorem~\ref{th:theorem6prop2}}

We establish the three items of the theorem. 

\xc{Proof of item~1.} 
Since the pair $(-J_{\mathsf{a}},C_{\mathsf{a}})$ is controllable and since $J_{\mathsf{a}}$ is Hurwitz,  $W_{-J_{\mathsf{a}},C_{\mathsf{a}}}(\infty)$ is a bounded matrix of full rank. 
Thus, there exists a constant $\alpha > 0$ such that $W_{-J_{\mathsf{a}},C_{\mathsf{a}}}(\infty) \geq \alpha I$. 
By item~1 of Proposition~\ref{prop:inthreeterms}, we can find a sufficiently large 
 $T_\alpha$  such that 
$$e^{J_{\mathsf{a}}T} (V_{\mathsf{s}}/V_{\mathsf{o}}) e^{J_{\mathsf{a}}^\dagger T} \geq \frac{\alpha}{2} I \quad \mbox{for all } T \geq T_\alpha,$$ 
and hence, 
$$V_{\mathsf{s}}/V_{\mathsf{o}} \geq \frac{\alpha}{2} e^{-J_{\mathsf{a}}T}e^{-J_{\mathsf{a}}^\dagger T} \quad \Rightarrow \quad
\lim_{T\to\infty}(V_{\mathsf{s}}/V_{\mathsf{o}})^{-1} \leq \lim_{T\to\infty}\frac{2}{\alpha} e^{J_{\mathsf{a}}^\dagger T}e^{J_{\mathsf{a}} T} = 0.$$

\xc{Proof of item~2.} 
 We write
\begin{equation}\label{eq:threeterms}
(V_{\mathsf{s}}/V_{\mathsf{o}})^{-1} V_{\mathsf{oa}}^\dagger V_{\mathsf{o}}^{-1}  = e^{J_{\mathsf{a}}^\dagger T} 
\left [ e^{J_{\mathsf{a}} T} (V_{\mathsf{s}}/V_{\mathsf{o}}) e^{J_{\mathsf{a}}^\dagger T} \right ]^{-1} \left [ e^{J_{\mathsf{a}}T} V_{\mathsf{oa}}^\dagger  V_{\mathsf{o}}^{-1} \right ].
\end{equation}
The first term on the right hand side  of~\eqref{eq:threeterms} is zero in the limit since $J_{\mathsf{a}}$ is Hurwitz. The second term is bounded by item~1 of Proposition~\ref{prop:inthreeterms}.
For the last term, we note that
\begin{equation}\label{eq:eJ22V23}
e^{J_{\mathsf{a}} T} V_{\mathsf{oa}}^\dagger V_{\mathsf{o}}^{-1}=  \widetilde V_{\mathsf{oa}}^\dagger  \widetilde V_{\mathsf{o}}^{-1} e^{J_{\mathsf{o}}T},
\end{equation}
which vanishes in the limit as $T \to \infty$ by item~2 of Proposition~\ref{prop:inthreeterms}.

\xc{Proof of item~3.} 
Recall that $V_{\mathsf{o}}(T) = W_{J_{\mathsf{o}}, C_{\mathsf{o}}}(T)$, with all the eigenvalues of $J_{\mathsf{o}}$ being imaginary.  
Thus, using Theorem~\ref{thm:pimlimit} and the arguments given in item~1 of Proposition~\ref{prop:inthreeterms} to show~\eqref{eq:v21-1},  we similarly obtain that $V_{\mathsf{o}}(T)^{-1}$ vanishes as $T\to\infty$. Next, we use~\eqref{eq:eJ22V23} to rewrite 
\begin{multline}\label{eq:thirdthreeterms}
V_{\mathsf{o}}^{-1} V_{\mathsf{oa}}(V_{\mathsf{s}}/V_{\mathsf{o}})^{-1} V_{\mathsf{oa}}^\dagger V_{\mathsf{o}}^{-1} \\ 
\hspace{2.2cm} = \left [V_{\mathsf{o}}^{-1} V_{\mathsf{oa}}e^{J^\dagger_{\mathsf{a}}T}\right]\left [e^{-J^\dagger_{\mathsf{a}}T}(V_{\mathsf{s}}/V_{\mathsf{o}})^{-1} e^{-J_{\mathsf{a}}T}\right ] \left [e^{J_{\mathsf{a}}T} V_{\mathsf{oa}}^\dagger V_{\mathsf{o}}^{-1} \right ]\\ 
=  \left [ e^{J_{\mathsf{o}}^\dagger T}   \widetilde V_{\mathsf{o}}^{-1} \widetilde V_{\mathsf{oa}} \right ]\left [ e^{J_{\mathsf{a}} T}(V_{\mathsf{s}}/V_{\mathsf{o}}) e^{J_{\mathsf{a}}^\dagger T} \right ]^{-1} \left [ \widetilde V_{\mathsf{oa}}^\dagger  \widetilde V_{\mathsf{o}}^{-1} e^{J_{\mathsf{o}}T} \right ].
\end{multline}
The first and the last terms on the right hand side of~\eqref{eq:thirdthreeterms} vanish as $T\to\infty$ by item~2 of Proposition~\ref{prop:inthreeterms}, and the second term is bounded by item~1 of Proposition~\ref{prop:inthreeterms}. \hfill $\qed$

\section{Proof of Theorem~\ref{thm:main}}\label{sec:pfgeneral}
Having dealt with the cases of imaginary and non-positive eigenvalues, we are now ready to prove  the general case of Theorem~\ref{thm:main}.
Recall that $J = P^{-1} A P$ is in  Jordan normal form, $C = P^{-1}B$, and the decompositions $J = \diag(J_{\mathsf{u}}, J_{\mathsf{s}}) $ and $C =[C_{\mathsf{u}}; C_{\mathsf{s}}]$, where the eigenvalues of $J_{\mathsf{u}}$ (resp. $J_{\mathsf{s}}$) have positive (resp. non-positive) real parts.

Consider the controllability Gramian associated with the pair $(J, C)$ and, for ease of notation, we use $V(T)$ as a shorthand notation:   
$$
V(T) := W_{J,C}(T) = \int_0^T e^{-J t} CC^\dagger e^{-J^\dagger t} \rrd t. 
$$

Similar to what has been done in the previous section, we partition the matrix $V(T)$ into $4$ blocks:
$$
V(T) = 
\begin{bmatrix}
V_{\mathsf{u}}(T) & V_{\mathsf{us}}(T) \\
V_{\mathsf{us}}(T)^\dagger & V_{\mathsf{s}}(T) 
\end{bmatrix},
$$
where
\begin{equation}\label{eq:defVij}
    \begin{aligned}
& V_{*}(T):= W_{J_*,C_*}(T) =  \int_0^T e^{-J_* t} C_* C^\dagger_* e^{-J^\dagger_* t} \rrd t , \quad \mbox{for } * = \mathsf{u},\mathsf{s}, \\
& V_{\mathsf{u}\mathsf{s}}(T):= \int_0^T e^{-J_{\mathsf{u}} t} C_{\mathsf{u}} C^\dagger_{\mathsf{s}} e^{-J_\mathsf{s}^\dagger t} \rrd t. 
\end{aligned}
\end{equation}
By definition, $\lim_{T\to\infty}V_{\mathsf{u}}(T) = W_{\hspace{-.04cm}\mathsf{u}}$, where $W_{\hspace{-.04cm}\mathsf{u}}$ is defined in~\eqref{eq:defW1}. 
Since $(J, C)$ is controllable, so are $(J_*, C_*)$ for $* = \mathsf{u},\mathsf{s}$. Thus, the matrices $V(T)$, $V_{\mathsf{u}}(T)$, and $V_{\mathsf{s}}(T)$ are nonsingular. 
Now, let $$V/V_{\mathsf{s}}:= V_{\mathsf{u}} - V_{\mathsf{us}}V_{\mathsf{s}}^{-1}V_{\mathsf{us}}^\dagger$$ 
be the Schur complement of $V_{\mathsf{s}}$ in $V$ (we again omit the dependence on $T$). 
The inverse of $V$ can be written explicitly as
\begin{equation}\label{eq:Vinverse}
V^{-1} = 
\begin{bmatrix}
(V/V_{\mathsf{s}})^{-1} & - (V/V_{\mathsf{s}})^{-1} V_{\mathsf{us}} V_{\mathsf{s}}^{-1} \\
- V_{\mathsf{s}}^{-1} V_{\mathsf{us}}^\dagger (V/V_{\mathsf{s}})^{-1}  & V_{\mathsf{s}}^{-1}+V_{\mathsf{s}}^{-1} V_{\mathsf{us}}^\dagger (V/V_{\mathsf{s}})^{-1} V_{\mathsf{us}} V_{\mathsf{s}}^{-1}
\end{bmatrix}.
\end{equation}

The proof of Theorem~\ref{thm:main} relies on the following technical result:

\begin{theorem}\label{thm:limit0}
The following items hold for the blocks $V_*(T)$, for $* = \mathsf{u},\mathsf{s},\mathsf{us}$, given in~\eqref{eq:defVij}: 
    \begin{enumerate}
    \item $\lim_{T\to\infty} V_{\mathsf{s}}(T)^{-1} = 0$. 
    \item $\lim_{T\to\infty} V_{\mathsf{us}}(T) V_{\mathsf{s}}(T)^{-1} = 0$.  
    \item $\lim_{T\to\infty} V_{\mathsf{us}}(T)V_{\mathsf{s}}(T)^{-1} V_{\mathsf{us}}(T)^\dagger = 0$. 
    \end{enumerate}
\end{theorem}

Item~1 of Theorem~\ref{thm:limit0} follows directly from Corollary~\ref{cor:theorem6prop2}; indeed, by its definition~\eqref{eq:defVij}, $V_{\mathsf{s}}(T)$ is the controllability Gramian of the controllable pair $(J_{\mathsf{s}}, C_{\mathsf{s}})$, and $J_{\mathsf{s}}$ has only eigenvalues with non-positive real parts.  
We establish the last two items of Theorem~\ref{thm:limit0} in the following subsections. Before that, we prove
Theorem~\ref{thm:main}, which is a corollary of Theorem~\ref{thm:limit0}:

\begin{proof}[Proof of Theorem~\ref{thm:main}]
We first use~\eqref{eq:Vinverse} to show that 
\begin{equation}\label{eq:limWjcTh1B}
\lim_{T\to\infty} V(T)^{-1} = \lim_{T\to \infty} W_{J,C}(T)^{-1}=
\begin{bmatrix}
W_{\hspace{-.04cm}\mathsf{u}}^{-1} & 0 \\
0 & 0 
\end{bmatrix}.
\end{equation}
By item~3 of Theorem~\ref{thm:limit0}, we have that
\begin{equation}\label{eq:limitvv2}
\lim_{T\to\infty} V/V_{\mathsf{s}} = \lim_{T\to\infty} \left ( V_{\mathsf{u}} - V_{\mathsf{us}}V_{\mathsf{s}}^{-1}V_{\mathsf{us}}^\dagger \right ) = \lim_{T\to\infty} V_{\mathsf{u}} = W_{\hspace{-.04cm}\mathsf{u}}.
\end{equation}
Using~\eqref{eq:limitvv2} and item~2, we have that every term in the following product exists in  the limit:
$$
\lim_{T \to \infty}(V/V_{\mathsf{s}})^{-1} V_{\mathsf{us}} V_{\mathsf{s}}^{-1} = W_{\hspace{-.04cm}\mathsf{u}}^{-1} \lim_{T \to \infty}  (V_{\mathsf{us}}V_{\mathsf{s}}^{-1}) = 0.
$$
Similarly, using~\eqref{eq:limitvv2} and items~1 and~2, we have that
\begin{align*}
\lim_{T\to \infty} \left ( V_{\mathsf{s}}^{-1}+V_{\mathsf{s}}^{-1} V_{\mathsf{us}}^\dagger (V/V_{\mathsf{s}})^{-1} V_{\mathsf{us}} V_{\mathsf{s}}^{-1} \right ) &= \lim_{T \to \infty}  V_{\mathsf{s}}^{-1} V_{\mathsf{us}}^\dagger W_{\hspace{-.04cm}\mathsf{u}}^{-1} V_{\mathsf{us}} V_{\mathsf{s}}^{-1}=0.
\end{align*}
Put together in~\eqref{eq:Vinverse}, the above yields~\eqref{eq:limWjcTh1B}. 
Theorem~\ref{thm:main} then follows from~\eqref{eq:limWjcTh1B} and the fact that $W_{A, B}(T)^{-1} = (P^{-1})^\dagger  W_{J, C}(T)^{-1} P^{-1}$.
\end{proof}

\subsection{Proof of item~2 of Theorem~\ref{thm:limit0}} 
We express $V_{\mathsf{us}}$ as
\begin{align*}
V_{\mathsf{us}}  = \int_0^T e^{-J_{\mathsf{u}} t} C_{\mathsf{u}} C^\dagger_{\mathsf{s}} e^{-J_\mathsf{s}^\dagger t} \rrd t  = 
\begin{bmatrix}
V_{\mathsf{uo}} & V_{\mathsf{ua}}
\end{bmatrix},
\end{align*}
where $V_{\mathsf{uo}}$ and $V_{\mathsf{ua}}$ are given by (using $C_{\mathsf{s}}^\dagger=\begin{bmatrix} C_{\mathsf{o}}^\dagger& C_{\mathsf{a}}^\dagger \end{bmatrix}$)
\begin{equation}\label{eq:defv3132}
\begin{aligned}
V_{\mathsf{uo}} & :=\int_0^T e^{-J_{\mathsf{u}} t} C_{\mathsf{u}} C_{\mathsf{o}}^\dagger e^{-J_{\mathsf{o}}^\dagger t} \rrd t, \\ 
V_{\mathsf{ua}} & :=\int_0^T e^{-J_{\mathsf{u}} t} C_{\mathsf{u}} C_{\mathsf{a}}^\dagger e^{-J_{\mathsf{a}}^\dagger t} \rrd t.
\end{aligned}
\end{equation}
Using the expression for $V_{\mathsf{s}}^{-1}$ given in~\eqref{eq:schurcompV2}, we have that
$$
V_{\mathsf{us}} V_{\mathsf{s}}^{-1} = 
\begin{bmatrix}
V_{\mathsf{uo}} & V_{\mathsf{ua}}
\end{bmatrix}
\begin{bmatrix}
 V_{\mathsf{o}}^{-1}+V_{\mathsf{o}}^{-1} V_{\mathsf{oa}} (V_{\mathsf{s}}/V_{\mathsf{o}})^{-1} V_{\mathsf{oa}}^\dagger V_{\mathsf{o}}^{-1} & -V_{\mathsf{o}}^{-1} V_{\mathsf{oa}} (V_{\mathsf{s}}/V_{\mathsf{o}})^{-1}   \\
- (V_{\mathsf{s}}/V_{\mathsf{o}})^{-1}  V_{\mathsf{oa}}^\dagger V_{\mathsf{o}}^{-1}  & (V_{\mathsf{s}}/V_{\mathsf{o}})^{-1}
\end{bmatrix}.
$$
We then decompose
\begin{equation*}
V_{\mathsf{us}} V_{\mathsf{s}}^{-1}  = 
\begin{bmatrix}
\left (V_{\mathsf{us}} V_{\mathsf{s}}^{-1}\right )_1 & \left (V_{\mathsf{us}} V_{\mathsf{s}}^{-1}\right )_2 
\end{bmatrix},
\end{equation*}
where 
\begin{align}\label{eq:twoterms}
\left (V_{\mathsf{us}} V_{\mathsf{s}}^{-1}\right )_1&:= V_{\mathsf{uo}}(V_{\mathsf{o}}^{-1}+V_{\mathsf{o}}^{-1} V_{\mathsf{oa}} (V_{\mathsf{s}}/V_{\mathsf{o}})^{-1} V_{\mathsf{oa}}^\dagger V_{\mathsf{o}}^{-1}) - V_{\mathsf{ua}} (V_{\mathsf{s}}/V_{\mathsf{o}})^{-1}  V_{\mathsf{oa}}^\dagger V_{\mathsf{o}}^{-1}, \\
\left (V_{\mathsf{us}} V_{\mathsf{s}}^{-1}\right )_2&:= - V_{\mathsf{uo}} V_{\mathsf{o}}^{-1} V_{\mathsf{oa}} (V_{\mathsf{s}}/V_{\mathsf{o}})^{-1} + V_{\mathsf{ua}} (V_{\mathsf{s}}/V_{\mathsf{o}})^{-1}. \label{eq:defv3v212p2}
\end{align}

We establish item~2 of Theorem~\ref{thm:limit0} by proving the following proposition:
\begin{proposition}\label{prop:item2thm6}
    The following items hold:
    \begin{enumerate}
    \item $\lim_{T\to\infty} \left(V_{\mathsf{us}} V_{\mathsf{s}}^{-1}\right )_1 =0$.
    \item $\lim_{T\to\infty} \left(V_{\mathsf{us}} V_{\mathsf{s}}^{-1}\right )_2 =0$.
    \end{enumerate}
\end{proposition}

As mentioned earlier, proving that the limits of the above products vanishes cannot be done by showing that each term in the products has a limit. Instead, we need to consider grouping  of terms and introduce one more buffer term. For our purpose here, the following term will be relevant:
\begin{equation}\label{eq:deftildev32}
\widetilde V_{\mathsf{ua}}(T):= V_{\mathsf{ua}} e^{J_{\mathsf{a}}^\dagger T }= \int_0^T e^{-J_{\mathsf{u}} t} C_{\mathsf{u}} C_{\mathsf{a}}^\dagger e^{J_{\mathsf{a}}^\dagger (T - t)} \rrd t.
\end{equation}
We have
\begin{lemma}\label{lem:tildev32}
It holds that $\lim_{T\to\infty} \widetilde V_{\mathsf{ua}}(T) = 0$.
\end{lemma}

\begin{proof}
Since $-J_{\mathsf{u}}$ and $J_{\mathsf{a}}$ are Hurwitz, there exist positive constants $c_1$, $c_1$, $\alpha_1$, and $\alpha_2$ such that
$$\|e^{-J_{\mathsf{u}} t} C_{\mathsf{u}} \|_F \leq c_1 e^{-\alpha_1 t} \quad \mbox{and} \quad \|e^{J_{\mathsf{a}} t} C_{\mathsf{a}} \|_F \leq c_2 e^{-\alpha_2 t},$$
where $\|\cdot \|_F$ is the Frobenius norm. 
Then, using sub-multiplicativity, we obtain that
\begin{equation*}
\|\widetilde V_{\mathsf{ua}}\|_F  \leq \int_0^T \|e^{-J_{\mathsf{u}} t} C_{\mathsf{u}} C_{\mathsf{a}}^\dagger e^{J_{\mathsf{a}}^\dagger (T - t)}  \|_F \, \rrd t \leq c_1c_2 e^{-\alpha_2 T}\int_0^T e^{(\alpha_2 - \alpha_1) t}  \rrd t.
\end{equation*}
It then follows that  
$$
\|\widetilde V_{\mathsf{ua}}\|_F \leq   
\begin{cases}
\frac{c_1c_2}{\alpha_2-\alpha_1} (e^{-\alpha_1 T}- e^{-\alpha_2 T}) & \mbox{if } \alpha_1 \neq \alpha_2, \\
c_1c_2 e^{-\alpha_2 T} T & \mbox{if } \alpha_1 = \alpha_2.
\end{cases}
$$
In either case, the norm converges to zero as $T \to \infty$ due to $\alpha_1,\alpha_2 >0$.
\end{proof}

With the lemma above, we will now establish Proposition~\ref{prop:item2thm6}.

\begin{proof}[Proof of Proposition~\ref{prop:item2thm6}]
We establish below the two items. 
\xc{Proof of item~1.} 
For the first term on the right hand side of~\eqref{eq:twoterms}, we first note that $V_{\mathsf{uo}}$, defined in~\eqref{eq:defv3132}, is bounded in the limit as $T \to \infty$ since $-J_{\mathsf{u}}$ is Hurwitz and $J_{\mathsf{o}}$ has imaginary eigenvalues (hence the entries of $e^{-J_{\mathsf{u}} t}$ decay exponentially fast while the entries of $e^{-J_{\mathsf{o}}t}$ grow  at most polynomially fast), so
\begin{equation}\label{eq:v31bounded}
\lim_{T \to \infty} \|V_{\mathsf{uo}}\| < \infty.
\end{equation}
Then, using item~3 of Theorem~\ref{th:theorem6prop2}, we conclude that the first term on the right hand side of~\eqref{eq:twoterms} vanishes as $T\to\infty$. For the second term, we have that 
\begin{align}
V_{\mathsf{ua}} (V_{\mathsf{s}}/V_{\mathsf{o}})^{-1}  V_{\mathsf{oa}}^\dagger V_{\mathsf{o}}^{-1} & = 
\left [ V_{\mathsf{ua}} e^{J_{\mathsf{a}}^\dagger T }\right ] \left [  e^{-J_{\mathsf{a}}^\dagger T} (V_{\mathsf{s}}/V_{\mathsf{o}})^{-1} e^{-J_{\mathsf{a}} T}  \right ] \left [e^{J_{\mathsf{a}} T}  V_{\mathsf{oa}}^\dagger V_{\mathsf{o}}^{-1} \right ] \notag\\
& = \widetilde V_{\mathsf{ua}} \left [  e^{-J_{\mathsf{a}}^\dagger T} (V_{\mathsf{s}}/V_{\mathsf{o}})^{-1} e^{-J_{\mathsf{a}} T}\right ] \left [\widetilde V_{\mathsf{oa}} \widetilde V_{\mathsf{o}}^{-1} e^{J_{\mathsf{o}}T} \right ] \notag \\
& = \widetilde V_{\mathsf{ua}} \left [ e^{J_{\mathsf{a}}T} (V_{\mathsf{s}}/V_{\mathsf{o}}) e^{J_{\mathsf{a}}^\dagger T}\right ]^{-1} \left [\widetilde V_{\mathsf{oa}} \widetilde V_{\mathsf{o}}^{-1} e^{J_{\mathsf{o}}T} \right ], \label{eq:fourththreeterms}
\end{align}
where the second equality follows from~\eqref{eq:deftildev32} and~\eqref{eq:eJ22V23}. By Lemma~\ref{lem:tildev32}, the first term of~\eqref{eq:fourththreeterms} vanishes as $T\to\infty$. By items~1 and~2 of Proposition~\ref{prop:inthreeterms}, the second term is bounded in the limit and the last term of~\eqref{eq:fourththreeterms} vanishes as $T\to\infty$. Thus, 
\begin{equation}\label{eq:322212321}
\lim_{T\to \infty} V_{\mathsf{ua}} (V_{\mathsf{s}}/V_{\mathsf{o}})^{-1}  V_{\mathsf{oa}}^\dagger V_{\mathsf{o}}^{-1} =0.
\end{equation}

\xc{Proof of item~2.} 
For the first term on the right hand side of~\eqref{eq:defv3v212p2}, note that $V_{\mathsf{uo}}$ is bounded as shown in~\eqref{eq:v31bounded} and that by item~2 of Theorem~\ref{th:theorem6prop2}, $V_{\mathsf{o}}^{-1} V_{\mathsf{oa}} (V_{\mathsf{s}}/V_{\mathsf{o}})^{-1}$ converges to $0$ as $T\to\infty$, so this entire term vanishes as well. 
For the second term, we write
\begin{multline}\label{eq:v32221-1}
V_{\mathsf{ua}} (V_{\mathsf{s}}/V_{\mathsf{o}})^{-1} = \left [V_{\mathsf{ua}} e^{J_{\mathsf{a}}^\dagger T} \right ] \left [ e^{-J_{\mathsf{a}}^\dagger T} (V_{\mathsf{s}}/V_{\mathsf{o}})^{-1} e^{-J_{\mathsf{a}} T}\right ] e^{J_{\mathsf{a}}T} \\= \widetilde V_{\mathsf{ua}} \left [ e^{J_{\mathsf{a}}T} (V_{\mathsf{s}}/V_{\mathsf{o}}) e^{J_{\mathsf{a}}^\dagger T}\right ]^{-1} e^{J_{\mathsf{a}}T}.
\end{multline}
By Lemma~\ref{lem:tildev32}, we have that the first term on the right hand side of~\eqref{eq:v32221-1} vanishes.  
By item~1 of Proposition~\ref{prop:inthreeterms}, the second term is bounded in the limit. Finally, since $J_{\mathsf{a}}$ is Hurwitz, the last term vanishes in the limit $T \to \infty$. 
\end{proof}

\subsection{Proof of item~3 of Theorem~\ref{thm:limit0}}
We show here that $\lim_{T\to\infty} V_{\mathsf{us}}(T)V_{\mathsf{s}}(T)^{-1} V_{\mathsf{us}}(T)^\dagger = 0$. 
Using again the expression for $V_{\mathsf{s}}^{-1}$ in~\eqref{eq:schurcompV2} and the decomposition of $V_{\mathsf{us}}$ in~\eqref{eq:defv3132}, we have that
\begin{align}
V_{\mathsf{us}}V_{\mathsf{s}}^{-1} V_{\mathsf{us}}^\dagger & = 
\begin{bmatrix}
V_{\mathsf{uo}}, V_{\mathsf{ua}}
\end{bmatrix}
\begin{bmatrix}
 V_{\mathsf{o}}^{-1}+V_{\mathsf{o}}^{-1} V_{\mathsf{oa}} (V_{\mathsf{s}}/V_{\mathsf{o}})^{-1} V_{\mathsf{oa}}^\dagger V_{\mathsf{o}}^{-1} & -V_{\mathsf{o}}^{-1} V_{\mathsf{oa}} (V_{\mathsf{s}}/V_{\mathsf{o}})^{-1}   \\
- (V_{\mathsf{s}}/V_{\mathsf{o}})^{-1}  V_{\mathsf{oa}}^\dagger V_{\mathsf{o}}^{-1}  & (V_{\mathsf{s}}/V_{\mathsf{o}})^{-1}
\end{bmatrix}
\begin{bmatrix}
V_{\mathsf{uo}}^\dagger \\ 
V_{\mathsf{ua}}^\dagger
\end{bmatrix} \notag \\
&= V_{\mathsf{uo}}\left[ V_{\mathsf{o}}^{-1}+V_{\mathsf{o}}^{-1} V_{\mathsf{oa}} (V_{\mathsf{s}}/V_{\mathsf{o}})^{-1} V_{\mathsf{oa}}^\dagger V_{\mathsf{o}}^{-1} \right] V_{\mathsf{uo}}^\dagger  \label{eq:1V323} \\
& \hspace{.4cm} - V_{\mathsf{uo}} \left[  V_{\mathsf{o}}^{-1} V_{\mathsf{oa}} (V_{\mathsf{s}}/V_{\mathsf{o}})^{-1} \right] V_{\mathsf{ua}}^\dagger \label{eq:2V323} \\
& \hspace{.4cm} - V_{\mathsf{ua}} \left [(V_{\mathsf{s}}/V_{\mathsf{o}})^{-1}  V_{\mathsf{oa}}^\dagger V_{\mathsf{o}}^{-1}  \right ] V_{\mathsf{uo}}^\dagger \label{eq:3V323} \\
& \hspace{.4cm} + V_{\mathsf{ua}} \left[ (V_{\mathsf{s}}/V_{\mathsf{o}})^{-1} \right] V_{\mathsf{ua}}^\dagger. \label{eq:4V323}
\end{align}
For the first term~\eqref{eq:1V323}, we note that $V_{\mathsf{uo}}$ (and hence, $V_{\mathsf{uo}}^\dagger$) is bounded in the limit as shown in~\eqref{eq:v31bounded} and that the expression in the bracket vanishes in the limit by item~3 of Theorem~\ref{th:theorem6prop2}.  Next, we have that the two terms in~\eqref{eq:2V323} and~\eqref{eq:3V323} are Hermitian conjugate. We see that~\eqref{eq:3V323} vanishes in the limit as $T \to \infty$ as it is the product of the term in~\eqref{eq:322212321} and $V_{\mathsf{uo}}^\dagger$ --- in the limit the former vanishes while the latter is again bounded. For the last term~\eqref{eq:4V323}, we have that
$$
V_{\mathsf{ua}} \left[ (V_{\mathsf{s}}/V_{\mathsf{o}})^{-1} \right] V_{\mathsf{ua}}^\dagger = \widetilde V_{\mathsf{ua}} \left [ e^{J_{\mathsf{a}}T} (V_{\mathsf{s}}/V_{\mathsf{o}}) e^{J_{\mathsf{a}}^\dagger T}\right ]^{-1}\widetilde V_{\mathsf{ua}}^\dagger,
$$
where $\widetilde V_{\mathsf{ua}}$ and $\widetilde V_{\mathsf{ua}}^\dagger$ vanish as $T\to\infty$ by Lemma~\ref{lem:tildev32} and the expression in the bracket is bounded by item~1 of Proposition~\ref{prop:inthreeterms}. \qed

\bibliographystyle{amsplain}
\bibliography{riccatibib.bib}
\end{document}